\documentclass[11pt]{article}
\def\x{2.4}
\usepackage[lmargin=\x cm,rmargin=\x cm,tmargin=\x cm,bmargin=\x cm,bindingoffset=0cm]{geometry}
\usepackage{mathscinet}
\usepackage[]{hyperref}
\hypersetup{colorlinks=true}

\usepackage{amsthm,mathdots}
\usepackage{nicefrac}
\usepackage{amsmath}
\usepackage{amssymb}
\usepackage{amsthm}
\usepackage{ams fonts}
\usepackage{mathrsfs}

\usepackage{graphicx} 
\usepackage{xcolor}
\usepackage{fancyhdr}
\usepackage[ansinew]{inputenc}
\usepackage{amsthm}
\usepackage{nicefrac}
\usepackage{amsmath}
\usepackage{amssymb}
\usepackage{amsthm}
\usepackage{amsfonts}
\usepackage{graphicx} 
\usepackage{tikz}
\usepackage{tikz-3dplot}
\usetikzlibrary{patterns}
\usetikzlibrary{matrix} 
\usepackage{multicol} 
\usepackage{float} 
\usepackage{makeidx} 
\usepackage{hyperref}
\usepackage[english]{babel}
\usepackage[nomessages]{fp}
\usepackage{mathtools}

\newtheorem{Theorem}{Theorem}[section]
\newtheorem{Proposition}[Theorem]{Proposition}
\newtheorem{Lemma}[Theorem]{Lemma}
\newtheorem{Corollary}[Theorem]{Corollary}

\newtheorem{Definition}[Theorem]{Definition}
\newtheorem{Example}[Theorem]{Example}

\newtheorem{Remark}[Theorem]{Remark}

\newcounter{IssueCounter}
\newtheorem{Issue}[IssueCounter]{Issue}

\newcounter{QuestionCounter}
\newtheorem{Question}[QuestionCounter]{Question}

\definecolor{color1}{rgb}{0.1,0.3,0.0}
\newcounter{ReminderCounter}
\newtheorem{Reminder}[ReminderCounter]{Reminder}

\newcommand{\mb}[1]{\mathbb{#1}}

\newcommand{\mf}[1]{\mathfrak{#1}}

\newcommand{\Aut}[2][]{\mathrm{Aut}_{#1}\!\left(#2 \right)}

\newcommand{\Ad}{\mathrm{Ad}}

\newcommand{\diag}{\mathrm{diag}}

\definecolor{color2}{rgb}{0.30,0.0,0.9}

\newcommand{\A}{\mathrm{Aut}}
\newcommand{\g}{{\mathfrak{g}}}
\newcommand{\ox}{\mathcal{O}_{\mathbb{X}}}
\newcommand{\CC}{{\mathbb{C}}}
\newcommand{\LL}{{\Lambda}}
\newcommand{\ZZ}{{\mathbb{Z}}}

\newcommand{\twp}{\tilde{\wp}}
\newcommand{\NN}{\mathbb{N}}

\newcommand{\ot}{\mathcal{O}_{\mathbb{T}}}
\newcommand{\TT}{{\mathbb{T}}}

\title{A classification of automorphic Lie algebras on complex tori}

\author{Vincent Knibbeler, Sara Lombardo, and Casper Oelen}

\begin{document}
\date{}
\maketitle

\begin{abstract}
We classify the automorphic Lie algebras of equivariant maps from a complex torus to $\mathfrak{sl}_2(\mathbb{C})$. For each case we compute a basis in a normal form. The automorphic Lie algebras correspond precisely to two disjoint families of Lie algebras parametrised by the modular curve of $\mathrm{PSL}_2(\ZZ)$, apart from four cases, which are all isomorphic to Onsager's algebra.
\end{abstract}

 \section{Introduction}
Automorphic Lie algebras are a class of infinite-dimensional Lie algebras over the complex field $\mathbb{C}$ that emerged in the context of mathematical physics, and more precisely in the context of integrable systems \cite{lombardo2004thesis,lombardo2004reductions,lombardo2005reduction}. They originated in the study of algebraic reduction of Lax pairs by Lombardo and Mikhailov \cite{lombardo2004thesis,lombardo2004reductions}, related to the notion of \emph{reduction groups}, proposed by Mikhailov in \cite{mikhailov1979integrability} and  \cite{mikhailov1980reduction}. The first appearance of a notion of automorphic Lie algebras is to be found in the PhD thesis by Lombardo, and subsequently in Lombardo and Mikhailov \cite{lombardo2005reduction}, where a systematic study of these algebras within the theory of integrability began. It then developed independently from this approach into the search and the description of invariant algebras, with contributions by Lombardo and Sanders, and later also Knibbeler \cite{lombardo2010classification, knibbeler2015invariants, knibbeler2017higher, knibbeler2020hereditary}. The subject of automorphic Lie algebras inspired recent research into group theory, by the authors \cite{knibbeler2023polyhedral}. Independently, equivariant map algebras were introduced and studied by representation theorists Neher, Savage and Senesi \cite{neher2012irreducible} in order to unify many Lie algebras that appear in mathematical physics. These included automorphic Lie algebras, although they were unaware of this at the time.\\

While a precise definition will be given in Section \ref{Setup}, for now, the reader should think of automorphic Lie algebras as Lie algebras of meromorphic maps (usually with prescribed poles) from a compact Riemann surface $X$, originally the Riemann sphere, into a finite dimensional Lie algebra $\mathfrak{g}$, which are equivariant with respect to a finite group $\Gamma$ acting on $X$ and on $\mathfrak{g}$, both by automorphisms. It is therefore not surprising that the study of these objects requires notions from algebra, geometry and analysis. The group $\Gamma$ plays the role of the reduction group in the context of integrable systems. In this paper, we refer to $\Gamma$ as the \emph{symmetry group} of an automorphic Lie algebra.\\

Considerable work on automorphic Lie algebras based on the Riemann sphere has been carried out in the past decades. More recently other Riemann surfaces have been investigated in \cite{duffield2024wild} where the foundations of a representation theory for automorphic Lie algebras have been developed. Automorphic Lie algebras in the context of modular forms have been investigated in \cite{knibbeler2023automorphic}. 
In \cite{duffield2024wild}, the authors find a local description of automorphic Lie algebras in the vicinity of a point on the compact Riemann surface. That is, the quotient of an automorphic Lie algebra obtained by truncating a local parameter is determined. However, it remains an open problem to have a global description of such Lie algebras on surfaces of positive genus, as we have for the Riemann sphere. First steps in this direction can be found in \cite{oelen2022automorphic}.\\

This paper develops those concepts further and presents a first classification in the case where $X$ is a genus 1 Riemann surface and when $\g=\mathfrak{sl}_2(\CC)$; we will refer to $\g$ as the \emph{base} Lie algebra. We will exclusively work over the complex numbers, for which reason we use the notation $\mf{sl}_2:=\mf{sl}_2(\CC)$.
 \\

To give a flavour of automorphic Lie algebras to a newcomer, let us start with a classic example. Consider a simple complex, finite dimensional Lie algebra $\mathfrak{g}$ and let $\sigma$ be an automorphism of order $N$ of $\mathfrak{g}$. Consider the space of Laurent polynomial maps $f:\mathbb{C}^* \rightarrow \mathfrak{g}$ which satisfy $f(\epsilon z)=\sigma f(z),$ where $\epsilon^N=1.$ Endow this space with the pointwise Lie bracket. This is an example of an automorphic Lie algebra with base Lie algebra $\mathfrak{g}$, and where the Riemann surface $X$ is the Riemann sphere $\mathbb{C}_{\infty}$. The space of meromorphic functions on $\mathbb{C}_{\infty}$ which are holomorphic outside $\{0,\infty\}$ is given by the space of Laurent polynomials $\mathbb{C}[z,z^{-1}]$. The group here is $C_N$ which acts on $\mathbb{C}_{\infty}$ as $z\mapsto \epsilon z$ and on $\mathfrak{g}$ by the automorphism $\sigma$, such that the equivariance condition reads $f(\epsilon z)=\sigma f(z)$. It is called a \emph{twisted loop algebra} and is probably best known for the role it plays in the construction of affine Kac-Moody algebras \cite{kac1990infinite}. Automorphic Lie algebras are natural generalisations of this type of algebra, in the sense illustrated below. \\

As was recognised only relatively recently \cite{roan1991onsager}, another important example of an automorphic Lie algebra can be traced back to the chemistry Nobel Prize winner, Lars Onsager, in his pioneering paper \cite{onsager1944crystal} on the exact solution of the planar Ising model.
This is not a loop algebra and is now known as the \emph{Onsager algebra}. It can be defined as the Lie algebra $\mathfrak{O}$ with complex basis $A_k$ and $G_m$ where $k\in \ZZ$ and $m\in \NN.$ The brackets are given by 
\begin{align*}
[A_k,A_l]&=4G_{k-l},\\
[A_k,G_m]&=2(A_{k-m}-A_{k+m}),\\
[G_m,G_n]&=0,
\end{align*}
with $G_{-m}= -G_m$ $(m>0)$ and $G_0=0$.
Onsager's work on the 2D Ising model was later simplified using different techniques, cf. \cite{kaufman1949crystal}, in which there was no appearance of the Onsager algebra.
This caused the Onsager algebra to be less studied until the early 1980s, when it began to appear in different contexts and in a different form. In a paper by Grady, \cite{dolan1982conserved} relations (the Dolan-Grady relations) between two nonlinear operators were found which guarantee the existence of infinitely many commuting charges. The Dolan-Grady relations were later found by Perk in \cite{perk1989star} to be connected to the Onsager algebra. For a concise overview, we refer to El-Cha\^{a}r \cite{el2012onsager}.
 
 In \cite[Theorem 2.5]{knibbeler2023automorphic} it is proven that the Onsager algebra $\mathfrak{O}$ is isomorphic to the Lie algebra $\mathfrak{A}=\CC\langle h,e,f\rangle \otimes_{\CC}\CC[x]$ with the Lie structure that is linear over polynomials in the indeterminate $x$ and  $$ [h,e]=2e,\quad [h,f]=-2f,\quad [e,f]=h\otimes x(x-1).$$
 This is the form we will encounter when proving that certain automorphic Lie algebras are isomorphic to the Onsager algebra.\\
 
The main result of this work is a classification of automorphic Lie algebras on complex tori and the construction of certain normal forms. We will present our classification theorem here and leave the construction of the normal forms, proving our classification theorem, for the main body of the paper.

Besides the Onsager algebra described above, we find two families of Lie algebras parametrised by the (open) modular curve. We denote these families by $\mathfrak{C}_{\tau}$ and $\mathfrak{S}_{\tau}$ and present them in terms of the traditional elliptic invariants $g_2$ and $g_3$ defined as
$$g_2(\tau)=60\sum_{\substack{a,b\in\mb Z\\ (a,b)\ne (0,0)}}(a+b\tau)^{-4},\quad g_3(\tau)=140\sum_{\substack{a,b\in\mb Z\\ (a,b)\ne (0,0)}}(a+b\tau)^{-6}$$ where $\tau$ is an element of the upper half plane $$\mathbb{H}=\{z\in \CC: \mathrm{Im}(z)>0\}.$$
The Lie algebra $\mathfrak{C}_{\tau}$ is the current algebra
$$
 \mathfrak{C}_{\tau}=\mf{sl}_2 \otimes_{\CC} \CC[x,y]/(y^2-4x^3+g_2(\tau)x+g_3(\tau)),
 $$
where, as usual, the Lie bracket is defined by extending the bracket of $\mf{sl}_2$ linearly over the polynomials in $x$ and $y$.
The second family is defined by
\begin{align}\label{def:Stau} \mathfrak{S}_{\tau}=\CC\langle E,F,H\rangle \otimes_{\CC} \CC[x],
\end{align} where we identify $X$ with $X\otimes 1$ for $X=E,F,H$,
with the Lie structure that is linear over polynomials in $x$ and satisfies 
$$ [H,E]=2E,\quad [H,F]=-2F,\quad [E,F]=H\otimes (4x^3-g_2(\tau)x-g_3(\tau)).$$ 

The elliptic invariants $g_2$ and $g_3$ are usually defined for any lattice $\Lambda$ in $\mb C$ and we could define families of Lie algebras $\mathfrak{C}_{\Lambda}$ and $\mathfrak{S}_{\Lambda}$ accordingly, but this does not produce any new Lie algebras: homothetic lattices produce isomorphic Lie algebras. Indeed, using the fact that $g_2(\alpha \Lambda)=\alpha^{-4}g_2(\Lambda)$ and $g_3(\alpha \Lambda)=\alpha^{-6}g_3(\Lambda)$ for any nonzero complex number $\alpha$, we can check that the linear extension of
$$A\otimes p(x,y) \mapsto A\otimes p(\alpha^{2}x,\alpha^{3}y)$$ descends to an isomorphism from $\mathfrak{C}_{\Lambda}$ to $\mathfrak{C}_{\alpha\Lambda}$. Likewise we can define an isomorphism from $\mathfrak{S}_{\Lambda}$ to $\mathfrak{S}_{\alpha\Lambda}$ with the assignments
$$
E \otimes p(x)\mapsto E\otimes \alpha^3 p(\alpha^{2}x), \quad F \otimes p(x)\mapsto F\otimes \alpha^3 p(\alpha^{2}x),\quad H \otimes p(x)\mapsto H\otimes p(\alpha^{2}x).
$$
Hence we can assume that the lattice $\Lambda$ is scaled to the canonical form $\mb Z\oplus \mb Z\tau$ for $\tau\in\mb H$ if we are only concerned with isomorphism classes of Lie algebras.

If $\tau$ is replaced by $\tau'=\frac{a\tau+b}{c\tau+d}$ for $a,b,c,d\in\mb Z$, $ad-bc=1$ we arrive at isomorphic Lie algebras again. Indeed, $\mb Z\oplus \mb Z\tau'=(c+d\tau)^{-1}\left(\mb Z\oplus \mb Z\tau\right)$, and we have just shown that homothetic lattices produce isomorphic Lie algebras. Thus, we only need to consider equivalence classes $[\tau]$ in the modular curve $\mathrm{SL}_2(\mb Z)\backslash\mb H$ if we are interested in $\mathfrak{C}_{\tau}$ and $\mathfrak{S}_{\tau}$ up to isomorphism, and we will occasionally write $\mathfrak{C}_{[\tau]}$ and $\mathfrak{S}_{[\tau]}$ accordingly.

We know, in fact, that $\mathfrak{C}_{\tau}\cong \mathfrak{C}_{\tau'}$ if and only if $[\tau]=[\tau']$ due to Fialowski and Schlichenmaier. In \cite[Proposition 4.7]{fialowski2005global} they show that if $\g$ is a semisimple finite-dimensional Lie algebra,  $\mathcal{A}$ and $\mathcal{B}$ two associative, commutative algebras (with units) and if the current algebras $\g\otimes_{\CC}\mathcal{A}$ and $\g\otimes_{\CC}\mathcal{B}$ are isomorphic as Lie algebras, then $\mathcal{A}$ and $\mathcal{B}$ are isomorphic as associative algebras. 
Write $\mathcal{A}_{\tau}=\CC[x,y]/(y^2-4x^3+g_2(\tau)x+g_3(\tau))$. 
From algebraic geometry, cf. \cite[Corollary 3.7]{hartshorne1977algebraic}, we learn that $\mathcal{A}_{\tau}\cong \mathcal{A}_{\tau'}$ implies $[\tau]=[\tau']$ and hence $\mathfrak{C}_{\tau}\cong  \mathfrak{C}_{\tau'}$ implies $[\tau]=[\tau']$. It is not known to the authors whether $\mathfrak{S}_{[\tau]}\cong \mathfrak{S}_{[\tau']}$ implies $[\tau]=[\tau']$.\\

Now that we know the Lie algebras $\mathfrak{C}_{\tau},\,\mathfrak{O}$ and $\mathfrak{S}_{\tau}$ we can present our classification result.
\begin{Theorem}\label{thm:classification}Let $\Gamma$ be a finite group acting faithfully on a complex torus $T$ with biholomorphic maps and also faithfully on $\mathfrak{sl}_2$ with Lie algebra isomorphisms.
Let $\mathbb{T}$ be the torus minus the orbit $\Gamma\cdot \{0\}$, and let $\mathbb{T}\rightarrow \mathbb{T}/\Gamma$ be the canonical projection onto the quotient Riemann surface. 
The isomorphism class of the automorphic Lie algebra consisting of $\Gamma$-equivariant holomorphic maps $\mathbb{T}\to\mf{sl}_2$, meromorphic at the punctures $\Gamma\cdot \{0\}$, is determined by the number of branch points of $\mathbb{T}\rightarrow \mathbb{T}/\Gamma$ as in Table \ref{TableIntro}. No other number of branch points occur.

The class $[\tau]$ in this table is the element of $\mathrm{SL}_2(\mb Z)\backslash\mb H$ corresponding to the torus $T/t(\Gamma)$, where $t(\Gamma)$ is the subgroup of $\Gamma$ generated by the elements that have no fixed point in $T$.
\end{Theorem}
\begin{table}[h!]
        \begin{center}
            \begin{tabular}{cc}
            $\#$ branch points &  Lie algebra       \\
            \hline\\[-4mm]
            0 & $\mathfrak{C}_{[\tau]}$  \\
            2 & $\mathfrak{O}_{\phantom{\tau}}$  \\
            3 & $\mathfrak{S}_{[\tau]}$
            \end{tabular}
        \caption{Lie algebra associated to the number of branch points of the quotient map $\mathbb{T}\rightarrow \mathbb{T}/\Gamma$.} 
        \label{TableIntro}
        \end{center}
         \end{table}
Here we use the well known one-to-one relation between complex tori up to isomorphism and elements of $\mathrm{SL}_2(\mb Z)\backslash\mb H$.

The dimension of the abelianisation $\mathfrak{A}/[\mathfrak{A},\mathfrak{A}]$ can be computed directly from the definition for $\mathfrak{A}=\mathfrak{C}_{\tau},\,\mathfrak{O},\,\mathfrak{S}_{\tau}$, and is then seen to be equal to the number of branch points of $\mathbb{T}\rightarrow \mathbb{T}/\Gamma$.
In particular, the Lie algebras $\mathfrak{C}_{\tau},\mathfrak{O}$, and $\mathfrak{S}_{\tilde{\tau}}$, where $\tau,\tilde{\tau}\in \mathbb{H}$, are pairwise nonisomorphic, and if two automorphic Lie algebras in our classification are isomorphic, then the number of branch points of the corresponding quotient maps are equal.  
  
 \section{Definitions, notations and first examples}\label{Setup} 
 
In this section we will define automorphic Lie algebras, introduce notations and present two examples: one where the Riemann surface is the Riemann sphere and one where it is a complex torus. 

Let $\g$ be a complex finite dimensional Lie algebra and let $X$ be a compact Riemann surface. Suppose $\Gamma$ is a finite group acting on $\g$ by Lie algebra isomorphisms and on $X$ by biholomorphisms. We denote the group of Lie algebra isomorphisms by $\A(\g)$ and the group of biholomorphisms of $X$ by $\A(X)$.  The group actions can be described by homomorphisms 
$$
\sigma:\Gamma\rightarrow \A(X),\quad
\rho:\Gamma\rightarrow \A(\g).
$$ 
Consider a finite $\Gamma$-invariant set $\mathcal{S}\subset X$ and denote by $\mathbb{X}$ the complement $X\setminus \mathcal{S}$. Thus $\mathbb{X}$ is a punctured compact Riemann surface.
We use the notation $\ox$ for the $\CC$-algebra of meromorphic functions on $X$ holomorphic on $\mathbb{X}$. Denote by $\tilde{\sigma}$ the induced homomorphism $\Gamma\rightarrow \A(\ox)$ defined by $\tilde{\sigma}(\gamma)f=f\circ \sigma(\gamma)^{-1}$ for $\gamma\in \Gamma$. 
A \emph{current algebra} is formed by taking the tensor product over $\CC$ of $\g$ and $\ox$. This has a natural structure of a Lie algebra by declaring $$\left[\sum_iA_i\otimes f_i, \sum_jB_j\otimes g_j\right]:=\sum_{i,j}[A_i,B_j]\otimes f_ig_j,$$ where $A_i,B_j\in \g$, $f_i,g_j\in \ox$ and the bracket on the right-hand side is the bracket in $\g$. We let $\Gamma$ act naturally on $\g\otimes_{\CC}\ox$ by the diagonal action 
\begin{equation}\label{DiagonalAction}
\gamma\cdot (A\otimes f)=\rho(\gamma)A\otimes \tilde{\sigma}(\gamma)f.
\end{equation} We will also write $\rho\otimes\tilde{\sigma}(\gamma)$ for the homomorphism corresponding to the action defined in \eqref{DiagonalAction}.\\

 An automorphic Lie algebra is the Lie subalgebra of the current algebra $\g\otimes_{\CC}\ox$ of those elements invariant under the specified action.
\begin{Definition}
The automorphic Lie algebra $\mathfrak{A}=\mathfrak{A}(\g, \mathbb{X},\Gamma,\rho,\sigma)$
is defined as the algebra of fixed points of the action of $\Gamma$ on $\g\otimes_{\CC}\ox$. That is,
$$\mathfrak{A}(\g, \mathbb{X},\Gamma,\rho,\sigma)=(\g\otimes_{\CC}\ox)^{\rho\otimes \tilde{\sigma}(\Gamma)}=\{a\in \g\otimes_{\CC}\ox: \gamma \cdot a =a \,\, \text{for}\,\, \text{any}\,\,\gamma\in \Gamma\}.$$
\end{Definition}
From now onwards, we will use the acronym \emph{aLia} for ``automorphic Lie algebra".
We will sometimes simply write $\mathfrak{A}=(\g\otimes_{\CC}\ox)^{\Gamma}$, suppressing the additional data, if they are clear from the context.\\

As mentioned in the introduction, the Onsager algebra is a Lie algebra that appears prominently in the context of aLias. Roan discovered in 1991 \cite{roan1991onsager} that $\mathfrak{O}$ is in fact an aLia \emph{avant la lettre}. We will give an example of an aLia, where $X$ is the Riemann sphere, which is isomorphic to the Onsager algebra. 
\begin{Example}[The simplest concrete form of the Onsager algebra]
Let $X=\CC_{\infty}$ be the Riemann sphere and let $\g=\mathfrak{sl}_2$. Take $\mathcal{S}=\{0,\infty\}$. The punctured surface $\mathbb{X}$ is given by ${X}\setminus\mathcal{S}\cong \CC^*$ and the associated algebra of functions is $\mathcal{O}_{\mathbb{X}}=\mathbb{C}[z,z^{-1}]$, the Laurent polynomials. Let $\Gamma$ be the group with two elements, generated by $\gamma$, and define the homomorphism $\sigma:\Gamma\rightarrow \A(X)$ by $\sigma(\gamma)z= z^{-1}$.
The function $$J=\frac{z+2+z^{-1}}{4}$$ is invariant under this action. It is in fact a Hauptmodul, in the sense that any invariant meromorphic function on the Riemann sphere is a rational function in $J$. This can be shown using the theory of divisors. If we then restrict the poles to $\{0,\infty\}$ we obtain $\mathbb{C}[z,z^{-1}]^{\Gamma}=\mathbb{C}[J]$. We denote this space by $\mathbb{C}[z,z^{-1}]^{+}$ as it is the $+1$ eigenspace of $\gamma$. The $-1$ eigenspace $\mathbb{C}[z,z^{-1}]^{-}$ is given by $\mathbb{C}[J](z-z^{-1})$. Indeed, if $f\in \mathbb{C}[z,z^{-1}]^{-}$ then it necessarily vanishes at $1$ and $-1$, hence $f/(z-z^{-1})$ is an element of $\mathbb{C}[z,z^{-1}]^{+}$.
\\
Define the homomorphism $\rho:\Gamma\rightarrow \A(\mathfrak{sl}_2)$ by $$
\rho(\gamma)\begin{pmatrix}
a & b\\
c &-a
\end{pmatrix}=\begin{pmatrix}
a & -b\\
-c &-a
\end{pmatrix}.
$$
Let $$e=\begin{pmatrix}
0 & 1\\
0 &0
\end{pmatrix},\quad f=\begin{pmatrix}
0 & 0\\
1 &0
\end{pmatrix},\quad h=\begin{pmatrix}
1 & 0\\
0 &-1
\end{pmatrix}.$$
This action on $\mathfrak{sl}_2$ yields the eigenspace decomposition $$\mathfrak{sl}_2^{+}=\CC h,\quad \mathfrak{sl}_2^{-}=\CC\langle e,f\rangle.$$
We have 
\begin{align*}
\mathfrak{A}&=(\mathfrak{sl}_2\otimes_{\CC} \mathcal{O}_{\mathbb{X}})^{\rho\otimes \tilde{\sigma}(\Gamma)}\\
&=\mathfrak{sl}_2^{+}\otimes_{\CC} \ox^{+}\oplus \mathfrak{sl}_2^{-}\otimes_{\CC} \ox^{-}\\
&=\CC h \otimes_{\CC}\CC[J] \oplus \CC\langle e,f\rangle \otimes_{\CC} \CC[J](z-z^{-1}).
\end{align*}
Define $$H=h\otimes 1,\quad E=e\otimes (z-z^{-1})/4,\quad F=f\otimes (z-z^{-1})/4$$
so that we can write $\mathfrak{A}= \CC\langle E,F,H\rangle\otimes_{\CC} \CC[J]$ with brackets $$[H,E]=2E, \quad [H,F]=-2F, \quad [E,F]=H\otimes J(J-1).$$
We have chosen the Hauptmodul $J$ with values $0$ and $1$ at the branch points ($-1$ and $1$), but we could have chosen any other two (distinct) values. They appear as zeros in the last Lie bracket.
\\[2mm]
This Lie algebra made a lot of appearances in the literature. Roan studied the Onsager algebra and found this aLia (in another basis) as a concrete form of the Onsager algebra which enabled him to classify its finite dimensional irreducible representations \cite{roan1991onsager}. 
The work of Lombardo and Sanders \cite{lombardo2010classification}, Knibbeler, Lombardo and Sanders \cite{knibbeler2014automorphic} and Bury and Mikhailov \cite{MR4187211} showed that any aLia of $\Gamma$-equivariant meromorphic maps ${\mb{C}_\infty}\to\mathfrak{sl}_2$ with precisely two $\Gamma$-orbits of ramification points in its holomorphic domain and precisely one $\Gamma$-orbit outside its holomorphic domain, is isomorphic to $\mathfrak{A}$. Knibbeler, Lombardo and Veselov realised this Lie algebra in terms of modular functions, and provided an explicit isomorphism $\mathfrak{O}\to\mathfrak{A}$ by $$A_0\mapsto H,\quad A_1\mapsto H\otimes (2J-1)-2E+2F$$
cf. \cite[Theorem 2.5]{knibbeler2023automorphic}.
\end{Example}

We will see in Section \ref{MainResults} that the Onsager algebra can also be realised as certain aLias on a complex torus with enough symmetry, that is, tori on which the groups $C_4$ and $C_6$ act. As opposed to the case of the Riemann sphere, the symmetry group $C_2$ will not appear in this context. Instead, the symmetry groups $C_3,C_4,C_6$ and $A_4$ are the only ones that give rise to aLias isomorphic to the Onsager algebra. \\

For our first example of an aLia on a complex torus, we will use the \emph{Weierstrass $\wp$-function}. This function is defined as follows. Let $\LL$ be a lattice and consider the complex torus $T=\CC/\LL$. The function $\wp_{\LL}:T\rightarrow\CC$ defined by 
\begin{align}\label{DefWp}\wp_\Lambda(z) = \frac{1}{z^2}+\sum_{0\neq\omega\in \Lambda}\left(\frac{1}{(z-\omega)^2}-\frac{1}{\omega^2}\right),
\end{align} is called the Weierstrass $\wp$-function associated with the lattice $\LL$. It is a meromorphic function on $T$ with poles precisely at the lattice points $\omega\in \LL$.
The reference to a lattice $\Lambda$ will be suppressed when there is no confusion. It is clear that $\wp(-z)=\wp(z)$ and $\wp'(-z)=-\wp'(z)$, where $\wp'$ denotes the derivative of $\wp$. Furthermore, $\wp$ satisfies the equation $(\wp')^2=4\wp^3-g_2\wp-g_3$ where $g_2,g_3\in \CC$ are the elliptic invariants. The space of even meromorphic functions on $T$ which are holomorphic on $T\setminus \{0\}$ is given by the polynomials in $\wp$ with complex coefficients, $\CC[\wp]$, cf. \cite[Proposition V.3.1]{busam2009complex}. Using this and the relation $(\wp')^2=4\wp^3-g_2\wp-g_3$, it is not difficult to see that any odd meromorphic function on $T$, holomorphic on $T\setminus \{0\}$, is an element of $\CC[\wp]\wp'$. Since any function on $T$ is a sum of an even and odd function, we have established the following fact.
\begin{Lemma}\label{RegularFunctions}
The algebra of meromorphic function on complex torus $T$, holomorphic on $T\setminus\{0\}$, is given by $\mathcal{O}_{T\setminus \{0\}}=\CC[\wp,\wp']$. \end{Lemma}
We now have all the ingredients to present our first, simple example of an aLia on a complex torus with symmetry group $C_2$. We will write $\mathbb{T}$ for $T\setminus \Gamma\cdot\{0\}$.
\begin{Example}\label{C2aLia}
Fix a complex torus $T=\CC/\LL$ with $\Lambda=\ZZ\oplus \ZZ\tau$. Let $\g=\mathfrak{sl}_2$ and $\Gamma =C_2$, generated by $\gamma$, where $\sigma:\Gamma\rightarrow \A(T)$ is the homomorphism defined by $\sigma(\gamma)z= -z.$ The punctured torus $\mathbb{T}$ is given by $T\setminus \{0\}$. Write $\wp=\wp_{\LL}$. Keeping in mind that $\mathcal{O}_{\mathbb{T}}=\CC[\wp,\wp']$ and $(\wp')^2=4\wp^3-g_2\wp-g_3$, it is easily seen that the algebra $\mathcal{O}_{\mathbb{T}}$ is given by $\mathcal{O}_{\mathbb{T}}=\mathcal{O}_{\mathbb{T}}^{+}\oplus \mathcal{O}_{\mathbb{T}}^{-}$, where $$\mathcal{O}_{\mathbb{T}}^{+}=\CC[\wp],\quad \mathcal{O}_{\mathbb{T}}^{-}=\CC[\wp]\wp'$$
denote the space of even and odd functions, respectively. 
Define the homomorphism $\rho:\Gamma\rightarrow \A(\mathfrak{sl}_2)$ by $$
\rho(\gamma)\begin{pmatrix}
a & b\\
c &-a
\end{pmatrix}=\begin{pmatrix}
a & -b\\
-c &-a
\end{pmatrix}.
$$
Let $$e=\begin{pmatrix}
0 & 1\\
0 &0
\end{pmatrix},\quad f=\begin{pmatrix}
0 & 0\\
1 &0
\end{pmatrix},\quad h=\begin{pmatrix}
1 & 0\\
0 &-1
\end{pmatrix}.$$
This action on $\mathfrak{sl}_2$ decomposes as $\mathfrak{sl}_2^{+}\oplus \mathfrak{sl}_2^{-}$, where $$\mathfrak{sl}_2^{+}=\CC h,\quad \mathfrak{sl}_2^{-}=\CC\langle e,f\rangle$$ are the eigenspaces with eigenvalue $1$ and $-1$, respectively. 
We have $$\mathfrak{A}=(\mathfrak{sl}_2\otimes_{\CC} \mathcal{O}_{\mathbb{T}})^{\rho\otimes \tilde{\sigma}(C_2)}=\mathfrak{sl}_2^{+}\otimes_{\CC} \ox^{+}\oplus \mathfrak{sl}_2^{-}\otimes_{\CC} \ox^{-}=\CC\langle h,e\otimes \wp',f\otimes \wp'\rangle \otimes_{\CC}\CC[\wp].$$
Define $$E=e\otimes \wp',\quad F=f\otimes \wp',\quad H=h\otimes 1.$$
Then $\mathfrak{A}=(\mathfrak{sl}_2\otimes_{\CC} \ox)^{\rho\otimes \tilde{\sigma}(C_2)}= \CC\langle E,F,H\rangle\otimes_{\CC} \CC[\wp]\cong \mathfrak{S}_{\tau}$ with brackets $$[H,E]=2E, \quad [H,F]=-2F, \quad [E,F]=H\otimes(4\wp^3-g_2(\tau)\wp-g_3(\tau)).$$ Observe that $(\mathfrak{sl}_2\otimes_{\CC} \mathcal{O}_{\mathbb{T}})^{\rho\otimes \tilde{\sigma}(C_2)}\not\cong \mathfrak{sl}_2\otimes_{\CC}\mathcal{O}_{\mathbb{T}}^{\tilde{\sigma}(C_2)}$, since the right-hand side is perfect, whereas the left-hand side is not.
\end{Example}

In the above example, $(\mathfrak{sl}_2\otimes_{\CC} \mathcal{O}_{\mathbb{T}})^{\rho\otimes \tilde{\sigma}(C_2)}\not\cong \mathfrak{sl}_2\otimes_{\CC}\mathcal{O}_{\mathbb{T}}^{\tilde{\sigma}(C_2)}$. This is a consequence of the fact that with this choice of homomorphism $\sigma$, the quotient $T/\sigma(C_2)$ has genus 0 or equivalently, the canonical projection $\pi:\TT\rightarrow \TT/\sigma(C_2)$ has ramification points. It is a general fact, cf. \cite{duffield2024wild}, that for a punctured compact Riemann surface $\mathbb{X}$ and a complex finite dimensional Lie algebra $\g$, we have $(\g\otimes_{\CC}\mathcal{O}_{\mathbb{X}})^{\Gamma}\not\cong \g\otimes_{\CC}\mathcal{O}_{\mathbb{X}}^{\Gamma}$ whenever $\mathbb{X}\rightarrow \mathbb{X}/\Gamma$ contains a ramification point $x_0$, when $\Gamma_{x_0}$ acts nontrivially on $\g$. We will later see how choosing a different homomorphism of $C_2\rightarrow \A(T)$ does yield an isomorphism.

\section{A classification scheme}
The aim of this section is to explain what we will precisely classify. Our objects of interest are aLias of the form $$\mf{A}(\mf{sl}_2,\TT,\Gamma,\rho,\sigma)=(\mf{sl}_2\otimes_{\CC}\ot)^{\rho\otimes\tilde{\sigma}(\Gamma)},$$ where $\Gamma$ is a finite group, $\rho:\Gamma\rightarrow \A(\mf{sl}_2)$ and $\sigma:\Gamma\rightarrow \A(T)$ homomorphisms, $\TT=T\setminus \mathcal{S}$, where $T$ is a complex torus and $\mathcal{S}$ an orbit of $\Gamma$ in $T$. We would like to determine the isomorphism classes of these aLias, which is the content of Theorem \ref{thm:classification}. We may assume that $\mathcal{S}$ is the orbit of $0$ because $\A(T)$ acts transitively on $T$ (using Lemma \ref{lem:IsomALias} below), and we may assume that $\rho$ and $\sigma$ are faithful, almost without losing generality, explained by Lemma \ref{NormalSubgroupAlia} and \ref{aLiaIdentity1}. These restrictions constitute our classification scheme.\\

We will first address some general aspects of invariant spaces with respect to some group action. Let $X$ be a compact Riemann surface, $\g$ a complex finite dimensional Lie algebra and let $\rho:\Gamma\rightarrow \A(\mathfrak{g})$ and $\sigma:\Gamma\rightarrow \A(X)$ be homomorphisms. Given a finite group $\Gamma$ and a $\Gamma$-module $M$, a standard technique for obtaining the space of invariants of the action of $\Gamma$ on $M$, denoted by $M^{\Gamma}$, is that of averaging over the group $\Gamma$. We will explain this method. Define the \emph{averaging operator} (also known as the \emph{Reynolds operator}) $\langle \cdot \rangle_{\Gamma}: M\rightarrow M$ by $$\langle m \rangle_{\Gamma}=\frac{1}{|\Gamma|}\sum_{\gamma\in \Gamma}\gamma\cdot m.$$ Clearly $\langle \cdot \rangle_{\Gamma}$ is linear if the action of $\Gamma$ on $M$ is linear. One sees that its image is $M^{\Gamma}$ and $\langle \langle m \rangle_{\Gamma} \rangle_{\Gamma}=\langle m \rangle_{\Gamma}$, thus $\langle \cdot \rangle_{\Gamma}$ is a projection onto $M^{\Gamma}$. Suppose $K$ is a normal subgroup of $\Gamma$, which we denote by $K\lhd\Gamma$. Then there is an obvious action of the quotient group $\Gamma/K$ on $M^{K}$ and we can check directly that $\langle \langle \cdot \rangle_{K} \rangle_{\Gamma/K}=\langle \cdot \rangle_{\Gamma}$. In particular, an aLia with symmetry group $\Gamma$ can be computed making use of the fact that it has a normal subgroup $K$.  Phrased in terms of our setting:
\begin{Lemma}\label{NormalSubgroupAlia}
Let $K\lhd \Gamma$. Then $(\g\otimes_{\CC}\mathcal{O}_{\mathbb{X}})^{\Gamma}=((\g\otimes_{\CC}\mathcal{O}_{\mathbb{X}})^{K})^{\Gamma/K}.$
\end{Lemma}
Let $\tilde{\rho}=\rho\otimes \tilde{\sigma}$ and $\tilde{K}=\ker\rho\cdot\ker\sigma$.
First of all, observe that we may replace $\Gamma$ by $\Gamma/\ker(\rho\otimes \tilde{\sigma})$, making use of the fact that $\ker(\rho\otimes \tilde{\sigma})$ is a normal subgroup of $\Gamma$ and invoking Lemma \ref{NormalSubgroupAlia}. Trivially, $(\g\otimes_{\CC} \mathcal{O}_{\mathbb{X}})^{\ker(\rho\otimes \tilde{\sigma})}=\g\otimes_{\CC} \mathcal{O}_{\mathbb{X}}$. We can therefore assume that $\ker\rho\cap\ker\tilde\sigma=1$ without losing generality. 
\begin{Lemma}\label{aLiaIdentity1}
$(\g\otimes_{\CC}\mathcal{O}_{\mathbb{X}})^{\rho\otimes\tilde{\sigma}(\Gamma)}=(\g^{\rho(\ker\sigma)}\otimes_{\CC}\mathcal{O}_{\mathbb{X}}^{\tilde{\sigma}(\ker\rho)})^{\tilde{\rho}(\Gamma)/\tilde{\rho}(\tilde{K})}.$
\end{Lemma}
\begin{proof}
Clearly $\tilde{\rho}(\tilde{K})\lhd \tilde{\rho}(\Gamma)$ since $\tilde{K}\lhd \Gamma$ and $\tilde{\rho}$ is a homomorphism. Suppose $A\in \g$ and $f\in \mathcal{O}_{\mathbb{X}}$. Now, if $\tilde{\rho}( \gamma_1 \gamma_2)\in \tilde{\rho}(\tilde{K})$ with $ \gamma_1\in \ker(\rho)$ and $ \gamma_2\in \ker(\sigma)$, then $$\tilde{\rho}( \gamma_1 \gamma_2)(A\otimes f)=\rho( \gamma_1 \gamma_2)\otimes \tilde{\sigma}( \gamma_1 \gamma_2)(A\otimes f)=\rho( \gamma_2)A\otimes \tilde{\sigma}( \gamma_1)f,$$
where the homomorphism $\rho\otimes\tilde{\sigma}$ is defined in \eqref{DiagonalAction}. Hence $(\g\otimes_{\CC}\mathcal{O}_{\mathbb{X}})^{\tilde{\rho}(\tilde{K})}=\g^{\rho(\ker\sigma)}\otimes_{\CC}\mathcal{O}_{\mathbb{X}}^{\tilde{\sigma}(\ker \rho)}.$ By employing Lemma \ref{NormalSubgroupAlia}, this proves the claim.
\end{proof}
Assume now the case of $X=T$ and denote again by $\mathbb{T}$ the punctured complex torus $T\setminus \sigma(\Gamma)\cdot \{0\}$ (see Section \ref{Setup}).
We know that $\mathcal{O}_{\mathbb{T}}^{\tilde{\sigma}(\ker\rho)}\cong\mathcal{O}_{\mathbb{T}/\sigma(\ker\rho)}.$ Now, $\mathbb{T}/\sigma(\ker\rho)$ is a punctured compact Riemann surface of genus 0 or 1. In the former case, we refer to previous literature \cite{lombardo2010classification, knibbeler2017higher, knibbeler2020hereditary, MR4187211}. In the latter case, we may assume that $\sigma(\ker \rho)=\tilde{\sigma}(\ker \rho)=1$. Together with $\ker\sigma\cap\ker\rho=1$ this is equivalent to $\ker\rho=1$. Now consider the homomorphism $\sigma$. If $\ker\sigma$ is nontrivial, the Lie algebra $\mf{sl}_2^{\rho(\ker\sigma)}$ has dimension less than 3 and therefore is abelian, and we will rule out this class.\\

The following simple lemma plays an important role in our classification. It appears in a more general setting in the context of equivariant map algebras in \cite{neher2012irreducible}. Recall that $\mathbb{X}$ stands for the punctured compact Riemann surface $X\setminus \mathcal{S}.$
\begin{Lemma}\label{lem:IsomALias}
Suppose $\rho:\Gamma\rightarrow \A(\mathfrak{g})$ and $\sigma:\Gamma\rightarrow \A(X)$ are homomorphisms and $\tilde{\rho}=\rho\otimes \tilde{\sigma}:\Gamma\rightarrow \A(\g\otimes_{\CC}\mathcal{O}_{\mathbb{X}})$. Let $\mathfrak{A}$ be the aLia defined by these actions: $$\mathfrak{A}=\{a\in \g\otimes_{\CC}\mathcal{O}_{\mathbb{X}}: \tilde{\rho}(\gamma)a=a \,\, \text{for}\,\, \text{any}\,\, \gamma\in \Gamma\},$$
where $a=A\otimes f$, with $A\in \g$ and $f\in \ox$.
Define a second action of $\Gamma$ on $\g$ and $X$ by $\rho'(\gamma)=\tau_1\rho(\gamma)\tau_1^{-1}$ and $\sigma'(\gamma)=\tau_2\sigma(\gamma)\tau_2^{-1}$, where $\tau_1\in \A(\g)$ and $\tau_2\in \A(X)$ such that $\tau_2(\mathcal{S})=\mathcal{S}$. Let $\tilde{\rho}'$ be given by $\tilde{\rho}'=\rho'\otimes \tilde{\sigma}'$ and define  $$\mathfrak{A}'=\{a\in \g\otimes_{\CC}\mathcal{O}_{\mathbb{X}}: \tilde{\rho}'(\gamma)a=a \,\, \text{for}\,\, \text{any}\,\, \gamma\in \Gamma\}.$$
Then $\mathfrak{A}\cong \mathfrak{A}'$ as Lie algebras.
\end{Lemma}
\begin{proof}
Define $\varphi: \g\otimes_{\CC}\ox\rightarrow \g\otimes_{\CC}\ox$ on simple tensors by $\varphi(A\otimes f)=\tau_1(A)\otimes \tilde{\tau}_2f$, where $A\in \g$, $f\in \ox$ and $\tilde{\tau}_2f(z)=f(\tau_2^{-1}z)$ for $z\in \mathbb{X}$, and extend $\varphi$ $\CC$-linearly to the whole space. It is a straightforward verification that $\varphi$ is Lie algebra isomorphism and that $\varphi\circ\tilde{\rho}(\gamma)=\tilde{\rho}'(\gamma)\circ \varphi$ for all $\gamma\in \Gamma$, that is, it intertwines $\tilde{\rho}$ and $\tilde{\rho}'$. This implies that $a\in \g\otimes_{\CC}\ox$ is invariant with respect to the action defined by $\tilde{\rho}$ if and only if $\varphi(a)$ is invariant with respect to the action defined by $\tilde{\rho}'$. Thus $\varphi$ restricts to an isomorphism between $\mf{A}$ and $\mf{A}'.$
\end{proof}
\section{Symmetry groups with $\g=\mathfrak{sl}_2$}
In this section we will investigate which symmetry groups $\Gamma$ will play a role in the confined context of our classification scheme.
This requires an understanding of the finite subgroups of $\A(\mathfrak{sl}_2)$ and of $\Aut{T}$.
\begin{Lemma}[Klein]\label{lem:Klein} Finite subgroups of $\A(\mathfrak{sl}_2)$ are classified by the list
 $$C_N,\quad D_N,\quad A_4,\quad S_4, \quad A_5.$$ 
If two finite subgroups of $\A(\mathfrak{sl}_2)$ are isomorphic, then they are conjugate.
\end{Lemma}
\begin{proof}
The group $\A(\mathfrak{sl}_2)$ is isomorphic to $\mathrm{PSL}_2(\CC)$ by the adjoint representation. A finite subgroup of $\mathrm{PSL}_2(\CC)$ leaves a Hermitian inner product invariant, and is therefore conjugate to a subgroup of $\mathrm{PSU}_2$ which in turn is isomorphic to $\mathrm{SO}_3$. Therefore, the statement of the lemma is equivalent to the analogue statement for the Lie group $\mathrm{SO}_3$ instead of $\A(\mathfrak{sl}_2)$. Klein showed that the subgroups of $\mathrm{SO}_3$ are precisely the orientation preserving isometries of $\mathbb{R}^3$ that fix the regular pyramids, regular polygons, regular tetrahedrons, regular octahedrons and regular icosahedrons, centred at the origin \cite{klein1956lectures}. This yields the groups listed in the statement, respectively. We may moreover assume that the vertices of the polyhedra have norm $1$. If we then take two polyhedra of the same type, there is an element of $\mathrm{SO}_3$ that transforms one to the other. This group element conjugates the symmetry groups of the two polyhedra, and proves the last statement of the lemma.
\end{proof}
\begin{Lemma}
\label{lem:groups up to isomorphism}
A finite group $G$ embeds in $\A(\mathfrak{sl}_2)$ and $\Aut{T}$ if and only if $G$ equals \begin{enumerate}
\item$C_N,\quad N\ge 1$
\item$D_N,\quad N\ge 2$
\item$A_4,\quad (g_2(\tau)=0)$
\end{enumerate}
\end{Lemma}
\begin{proof}
By Lemma \ref{lem:Klein}, we know that the finite subgroups of $\A(\mathfrak{sl}_2)$ are given by $$C_N,\quad D_N,\quad A_4,\quad S_4, \quad A_5.$$ We will show that only the subgroups as given in the statement, are simultaneously subgroups of $\A(T)$ as well. 

Let $T$ be a complex torus. It follows from \cite[Proposition III.1.11]{miranda1995algebraic} that any automorphism $\sigma$ of $T$ is of the form $\sigma(z)=\epsilon z+\alpha$, where $\epsilon$ is a suitable root of unity (for which $T$ has multiplication by $\epsilon$) and $\alpha\in T$. This immediately gives a semi-direct product structure $\mathrm{Aut}(T)=\A_0(T) \ltimes t(T)$ where $\A_0(T)$ is the group of automorphisms fixing zero and $t(T)$ is the group of translations of $T$. Now, take two automorphisms $\sigma,\sigma'$ defined by $\sigma(z)=\epsilon z+\alpha$ and $\sigma'(z)=\epsilon'z+\alpha'$. Then $[\sigma,\sigma'](z)=z-\alpha'-\epsilon'\alpha+\epsilon\alpha'+\alpha$, so that $[\sigma,\sigma']$ is indeed a translation. Hence $[\mathrm{Aut}(T),\mathrm{Aut}(T)]\subset t(T)$.

 Now, let $r\in t(T)$ be any translation of $T$, say $r(z)=z+\alpha$. We will show that there are $s\in  \mathrm{Aut}_0(T)$ and $r'\in  \mathrm{Aut}(T)$ such that $r=[s,r']$. Suppose $s(z)=-z$ (any torus has multiplication by $-1$) and let $r'(z)=z-\frac{\alpha}{2}$. Then $sr'(z)=-z+\frac{\alpha}{2}$. Hence $[s,r'](z)=-(-z-\frac{\alpha}{2}-\frac{\alpha}{2})=z+\alpha=r(z)$. This proves the claim that $[\A(T),\A(T)]=t(T)$.

For the proof that $C_N$, $D_N$ are subgroups of $\A(T)$, as well as $A_4$ for a suitable torus $T$, we refer to Lemma \ref{lem:groups up to conjugation} below.

Let us now argue that $S_4$ and $A_5$ are not subgroups of $\A(T)$ for any complex torus $T$. To see that $S_4$ is not a subgroup of $\A(T)$, note that $[S_4,S_4]=A_4$, which is nonabelian, whereas for $\Gamma\subset \A(T)$, the commutator subgroup $[\Gamma,\Gamma]\subset t(T)$ is abelian.

 Finally, suppose for a contradiction that $A_5\subset \A(T)$ for some complex torus $T$. Then $A_5\cong C_{\ell}\ltimes H$ for some normal subgroup $H\neq 1$, since $\A(T)=\A_0(T)\ltimes t(T)$. However, $A_5$ is simple and thus it cannot have a proper normal subgroup. This shows that the list above is all there is in the intersection of finite subgroups of $\A(\mathfrak{sl}_2)$ and $\A(T)$, for all complex tori $T$.
\end{proof}
\begin{Remark}
Observe that we can write the groups from Lemma \ref{lem:groups up to isomorphism} in terms of semidirect products as follows: $$1\ltimes C_N, \quad C_2\ltimes C_N,\quad C_{3}\ltimes (C_2\times C_2),\quad C_{\ell}\ltimes 1,$$
where $\ell\in \{1,2,3,4,6\}$ and $N\in \NN.$ When we write $G\ltimes K\subset \A_0(T)\ltimes t(T)$, we shall tacitly assume that $G\subset \A_0(T)$ and $K\subset t(T)$.
\end{Remark}
 \begin{Lemma}
        \label{lem:groups up to conjugation}
        The subgroups of $\Aut{T}$ 
        which are isomorphic to 
        one of the finite groups of Lemma \ref{lem:groups up to isomorphism} 
        are classified by the following list,  up to conjugation.
        \begin{enumerate}
            \item $C_N=\langle r: r^N=1\rangle$, 
            \begin{center}
            \begin{enumerate}
                \item 
                $
                C_{\ell}\subset \A_0(T),\quad r(z)=e^{ 2\pi i/\ell}z\quad (\ell\in\{2,3,4,6)\}).
               $ \label{item1A}
                \item 
                $
                C_N \subset t(T),\quad r(z)=z+\alpha \quad (\alpha\,\,\text{is a}\,\, N\text{-torsion point in }T).
                $ \label{item1B}
            \end{enumerate}
            \end{center}
            \item $D_N=\langle s,r:s^2=r^N=1, \, (sr)^2=1\rangle$,
            \begin{center}
            \begin{enumerate} 
                \item
                $
                C_2\times C_2\subset t(T),\quad s(z)=z+\tau/2, \quad r(z)=z+1/2.
                $ \label{item2A}
                \item
               $
                C_2\ltimes C_N\subset \A_0(T)\ltimes t(T),
                \quad s(z)=-z, 
                \quad r(z)=z+\alpha
                $\\
                $(\alpha\,\,\text{is a}\,\, N\text{-torsion point in }T)$. \label{item2B}
            \end{enumerate}
            \end{center}
            \item
            $A_4=\langle s, r_1,r_2: s^3=r_1^2=r_2^2=1,\, sr_1s^{-1}=r_1r_2=r_2r_1,\, sr_2s^{-1}=r_1\rangle$, $\tau=e^{2\pi i/3}$,
                \label{item3}
            \begin{description}
            \item $s(z)=e^{2\pi i/3}z,\quad r_1(z)= z+1/2.$
            \end{description}
        \end{enumerate}
    \end{Lemma}
    Here we recall that the case \ref{item1A} with $\ell=3,4$ or $6$ 
    and the case \ref{item3} only occurs for special tori, 
    cf. \cite[Proposition III.1.12]{miranda1995algebraic},
    and the other cases occur in any torus.
    \begin{proof} 
        We leave it to the reader to verify that the subgroups in the statements satisfy the group relations and 	thus are indeed of the mentioned isomorphism class.

        We start with the cyclic groups. 
        If an element $r$ of $\A(T)$ fixes an element $p$ of $T$, 
        then $t(p) r t(p)^{-1}\in \A_0(T)$ (where $t(p)(z)=z+p$)
        and $t(p) \langle r\rangle t(p)^{-1}$ is as described in \ref{item1A}.
        If on the other hand $r$ has no fixed points, 
        then it is as described in \ref{item1B}.
    
        Now for the dihedral groups, 
        consider first the abelian case $D_2=C_2\times C_2$. 
        This group occurs precisely once in $t(T)$, as in \ref{item2A}, 
        since both groups have precisely 3 elements of order $2$.

        Suppose now that $D_2\subset \A(T)$ has an element $s$ 
        which is not contained in $t(T)$. 
        Using a conjugation as we did to classify the cyclic groups, 
        we may assume that $s\in\A_0(T)$, so that $s(z)=-z$.
        There must also be a nontrivial element of $D_2$, say $r$, 
        contained in $t(T)$.
        Indeed, the elements of order $2$ in $\A(T)\setminus t(T)$ 
        are the maps $z\mapsto -z+b$. A product of two such maps is in $t(T)$.
        Thus the group $D_2$ is as described in \ref{item2B} with $N=2$.
        
        For the remaining dihedral groups $D_N\subset\A(T)$ ($N\ge 3$) 
        we notice that $[D_N,D_N]=\langle r^2\rangle\subset [\A(T),\A(T)]\subset t(T)$
        implies that $r\in t(T)$.
        At least one of the order $2$ elements $D_N\setminus\langle r \rangle$
        must be outside of $t(T)$ because $D_N$ is nonabelian. 
        Taking a conjugate of the group, we may again assume that such an element,
        say $s$, fixes $0$. Thus we arrive at the remaining groups 
        described in \ref{item2B}.

        Finally, for the group $A_4$ we argue as follows.
        The derived subgroup 
        $[A_4,A_4]=\langle r_1,r_2\rangle\subset [\A(T),\A(T)]\subset t(T)$
        is uniquely determined by the only $3$ elements in $t(T)$ of order $2$.
        The element $s$ must be outside of $t(T)$ for $A_4$ is nonabelian.
        Taking a conjugate of the group, we may assume that $s$ fixes zero, 
        which leaves two options: $s(z)=e^{2\pi i/3}z$ and $s(z)=e^{4\pi i/3}z$. 
        Both options generate the same group, 
        since there is an automorphism of $A_4$
        sending $s$ to $s^2$.
    \end{proof}
    \begin{Remark}
        The lemma above gives a classification of subgroups 
        $\Gamma\subset\A(T)$.
        One may be interested in a classification of embeddings 
        (injective homomorphisms)
        $\Gamma\to\A(T)$ instead. 
        The difference is in the consideration of $\A(\Gamma)$:
        two embeddings $\Gamma\to\A(T)$ have the same image 
        if and only if one is the composition of the other with 
        an automorphism of $\Gamma$.
    \end{Remark}

\begin{Lemma}\label{lem:branch points}
Let $\Gamma\subset \A_0(T)\ltimes t(T)$ be a finite subgroup and $\TT=T\setminus \Gamma\cdot\{0\}$. Then for the canonical projection $\pi:\TT\rightarrow \TT/\Gamma$, we have $$\#(\text{branch points of }\pi:\TT\rightarrow \TT/\Gamma)=\begin{cases} 0\quad \text{if}\,\, \Gamma\subset t(T),\\
2\quad \text{if}\,\, \Gamma=C_{\ell}\ltimes 1\,\, \text{or} \,\, \Gamma=C_3\ltimes (C_2\times C_2),\\
3\quad \text{if}\,\, \Gamma=C_2\ltimes C_N,
\end{cases}$$
where $\ell\in \{3,4,6\}$.
\end{Lemma}
\begin{proof}
A point $p\in T$ is a ramification point of $\pi$ if the multiplicity of $\pi$ at $p$, denoted by $\mathrm{mult}_p(\pi)$, is at least 2. By \cite[Theorem III.3.4]{miranda1995algebraic}, $\mathrm{mult}_p(\pi)$ equals $|\Gamma_p|$, where $\Gamma_p=\{\gamma\in\Gamma: \gamma\cdot p=p\}$. It is clear there are no branch points if $\Gamma\subset t(T)$ since a translation has no fixed points. Let now $\Gamma=C_3\subset \A_0(T)$ where $T=\CC/\ZZ\oplus \ZZ\omega_6$. It is straightforward to check that $|\Gamma_p|>1$ if and only if $p\in \{0,1/2,\omega_6/2,(1+\omega_6)/2\}$. The only points that are in the same $\Gamma$-orbit, are $1/2$ and $\omega_6/2$. Hence the total number of branch points of $\pi$ after deleting $\Gamma\cdot \{0\}$ equals 2. The other cases follow by similar arguments.
\end{proof}

\section{Functional aspects of aLia on complex tori}
We will now discuss the functional aspects of aLias on complex tori. Since we are considering meromorphic $\mf{sl}_2$-valued maps on a complex torus, a starting point is to understand meromorphic functions on a complex torus. There are multiple (equivalent) approaches to this, e.g. via \emph{Jacobi theta-functions} or \emph{Weierstrass functions}. We shall further develop the approach with the Weierstrass $\wp$-function, as we have started to do before Example \ref{C2aLia}.\\

Below we will formulate some elementary properties of the Weierstrass $\wp$-function. Define the values of the half lattice points of a given lattice $\LL_{\tau}=\ZZ\oplus \ZZ\tau$, under $\wp_{\LL_{\tau}}$: \begin{align}\label{HalfLatticePoints}
e_1=\wp_{\LL_{\tau}}(1/2),\quad e_2=\wp_{\LL_{\tau}}(\tau/2), \quad e_3=\wp_{\LL_{\tau}}((1+\tau)/2).\end{align}
Observe that $\wp_{\LL_{\tau}}'$ vanishes at the half lattice points, since $\wp_{\LL_{\tau}}'(-z)=-\wp_{\LL_{\tau}}'(z)$.
 
As we have remarked before Example \ref{C2aLia}, there is the relation between the square of the derivative $\wp'$ and $\wp$ itself, where we now write $\wp$ without specifying the lattice: \begin{align}\label{wpRel}
(\wp')^2=4\wp^3-g_2\wp-g_3.
\end{align}
Because $\wp'$ vanishes at the half lattice points, the right-hand side of \eqref{wpRel} factors as $4(\wp-e_1)(\wp-e_2)(\wp-e_3)$. From this, one gets the following relations between the $e_i$: $$e_1+e_2+e_3=0,\quad e_1e_2+e_1e_3+e_2e_3=-g_2/4,\quad e_1e_2e_3=g_3/4.$$
The next basic lemma records how Weierstrass $\wp$-functions associated to homothetic lattices are related.
\begin{Lemma}\label{scaling lemma}
For $\alpha\in\mathbb{C}^{*}$, we have $\wp_{\alpha\Lambda}(z)=\alpha^{-2}\wp_{\Lambda}\left(\alpha^{-1}z\right).$
\end{Lemma}
We write out explicitly two special cases of Lemma \ref{scaling lemma} that will play an important role in the course of the paper. These cases correspond to the square lattice $\LL_i=\ZZ\oplus \ZZ i$ and the hexagonal lattice $\LL_{\omega_6}=\ZZ\oplus \ZZ\omega_6$, which are the only lattices (up to homothety) that satisfy the property $\alpha \LL=\LL$ for some $\alpha\neq \pm 1$.
\begin{align}
\wp_{\LL_i}(i^{-1}z)&=-\wp_{\LL_i}(z), &\wp_{\LL_i}'(i^{-1}z)&=-i\wp_{\LL_i}'(z),\label{wpScaling1}\\
\wp_{\LL_{\omega_6}}(\omega_6^{-1}z)&=\omega_6^{2}\wp_{\LL_{\omega_6}}(z),&\wp_{\LL_{\omega_6}}'(\omega_6^{-1}z)&=-\wp_{\LL_{\omega_6}}'(z)\label{wpScaling2}.
\end{align}
The next lemma describes the isotypical components of the action $C_{\ell}\subset \A_0(T)$ on $\mathcal{O}_{\mathbb{T}}$, where $\ell\in \{2,3,4,6\}$ and $T$ is a suitable complex torus. These decompositions will be used in Theorem \ref{ClaLias}, where we construct a basis for aLias with symmetry group $C_{\ell}$. We denote by $\chi_j$, $j=0,\ldots,\ell-1$ the characters of the group $C_{\ell}$, which are defined as $\chi_j(r^k)=\omega_\ell^{jk}$, where $r$ generates $C_{\ell}$.
\begin{Lemma}\label{IsotypicalComponents}
Let $\sigma_{\ell}:C_{\ell}\rightarrow \A_0(T)$ be the homomorphism $\sigma_{\ell}(r)(z)=\omega_{\ell}z$ and let $\wp:=\wp_{\LL}$. Suppose that $\mathcal{S}=\{0\}$ and $\mathbb{T}=T\setminus \mathcal{S}$.
Then for $\Gamma=C_2$ we have
$$
\mathcal{O}_{\mathbb{T}}^{\chi_0}=\CC[\wp],\quad \mathcal{O}_{\mathbb{T}}^{\chi_1}=\CC[\wp]\wp'.
$$
For $\Gamma=C_3$ we have
$$
\mathcal{O}_{\mathbb{T}}^{\chi_0}=\CC[\wp'],\quad \mathcal{O}_{\mathbb{T}}^{\chi_1}=\CC[\wp']\wp^2,\quad \mathcal{O}_{\mathbb{T}}^{\chi_2}=\CC[\wp']\wp.
$$
For $\Gamma=C_4$ we have
\begin{align*}
\mathcal{O}_{\mathbb{T}}^{\chi_0}&=\CC[\wp^2], &
\mathcal{O}_{\mathbb{T}}^{\chi_1}&=\CC[\wp^2]\wp \wp',\\
\mathcal{O}_{\mathbb{T}}^{\chi_2}&=\CC[\wp^2]\wp, &
\mathcal{O}_{\mathbb{T}}^{\chi_3}&=\CC[\wp^2]\wp'.
\end{align*}
For $\Gamma=C_6$ we have 
\begin{align*}
\mathcal{O}_{\mathbb{T}}^{\chi_0}&=\CC[\wp^3], &
\mathcal{O}_{\mathbb{T}}^{\chi_1}&=\CC[\wp^3]\wp^2 \wp',\\
\mathcal{O}_{\mathbb{T}}^{\chi_2}&=\CC[\wp^3]\wp, &
\mathcal{O}_{\mathbb{T}}^{\chi_3}&=\CC[\wp^3]\wp',\\
\mathcal{O}_{\mathbb{T}}^{\chi_4}&=\CC[\wp^3]\wp^2,&
\mathcal{O}_{\mathbb{T}}^{\chi_5}&=\CC[\wp^3]\wp\wp'.
\end{align*}
\end{Lemma}
\begin{proof}
We will only prove the claims for $\Gamma=C_3$; the rest follows in the same way.
We know by Lemma \ref{RegularFunctions} that $\mathcal{O}_{\mathbb{T}}=\mathbb{C}[\wp,\wp']$.
Using (\ref{wpScaling2}) we see that $\CC[\wp']\subset \mathcal{O}_{\mathbb{T}}^{\chi_0}$, $\CC[\wp']\wp^2\subset \mathcal{O}_{\mathbb{T}}^{\chi_1}$ and $\CC[\wp']\wp\subset \mathcal{O}_{\mathbb{T}}^{\chi_2}.$
It remains to show that $\mathbb{C}[\wp,\wp']$ is a subset of $\CC[\wp']\oplus \CC[\wp']\wp^2\oplus \CC[\wp']\wp$. If $\wp^a\wp'^b$ is a monomial with $a\ge 3$, then we can substitute $\wp^3=\frac{1}{4}(\wp')^2+\frac{g_2}{4}\wp+\frac{g_3}{4}$, using \eqref{wpRel}, to obtain a polynomial in $\wp$ and $\wp'$ where all exponents of $\wp$ are less than $a$. Repeating this substitution finitely many times, we see that $\wp^a\wp'^b$ is an element of $\CC[\wp']\oplus \CC[\wp']\wp^2\oplus \CC[\wp']\wp$. Hence $\mathbb{C}[\wp,\wp']$ is indeed a subset of $\CC[\wp']\oplus \CC[\wp']\wp^2\oplus \CC[\wp']\wp$.
\end{proof}
ALias on complex tori with symmetry group $\Gamma$ and base Lie algebra $\mathfrak{sl}_2$, as in our classification, turn out to be generated by three generators over their algebra of invariants, as we will establish in Section \ref{MainResults}. We would therefore like to know what the algebra of invariants $\mathcal{O}_{\mathbb{T}}^{\Gamma}$ is, to obtain explicit forms of the Lie algebras. For aLias with symmetry group $C_N\subset t(T)$, this is described in the following lemma. For $\alpha\in T$ of finite order, let $\LL_{(\alpha)}=\ZZ+\ZZ\alpha+\ZZ\tau$ and $T_{(\alpha)}=\CC/\LL_{(\alpha)}$.
\begin{Lemma}\label{lem:CN invariants}
Let $\sigma_{\alpha}:C_N\rightarrow \A(T)$ be given by $\sigma_{\alpha}(r)(z)=z+\alpha$. Then $$\mathcal{O}_{\mathbb{T}}^{\tilde{\sigma}_{\alpha}(C_N)}= \CC[\wp_{\Lambda_{(\alpha)}}, \wp_{\Lambda_{(\alpha)}}'],$$
\end{Lemma}
\begin{proof}
We know that $\mathcal{O}^{\tilde{\sigma}_\alpha(C_N)}_{\mathbb{T}}$ is the space of $\Lambda_{(\alpha)}$-periodic meromorphic functions which are holomorphic on $T\setminus \Lambda_{(\alpha)}$. With this perspective, $\mathcal{O}_{\mathbb{T}}^{\tilde{\sigma}_{\alpha}(C_N)}=\mathcal{O}_{\mathbb{T}/\sigma_{\alpha}(C_N)}$. Now use Lemma \ref{RegularFunctions} with $T$ taken to be the complex torus $T/\sigma_{\alpha}(C_N)$. Thus $\mathcal{O}_{\mathbb{T}}^{\tilde{\sigma}_{\alpha}(C_N)}=\mathbb{C}[\wp_{\Lambda_{(\alpha)}}, \wp_{\Lambda_{(\alpha)}}'].$ \end{proof}
We will now construct meromorphic functions on complex tori that will play a fundamental role throughout, especially in the construction of normal forms. Given a meromorphic function $f$ on $T$, the \emph{divisor} $(f)$ of $f$ is defined as the following formal $\ZZ$-linear combination of points $p\in T$, $(f)=\sum_{p\in T}\mathrm{ord}_p(f)(p)$, where $\mathrm{ord}_p(f)$ is the order of a zero or pole of $f$ at $p$ (if $p$ is neither a zero nor a pole, $\mathrm{ord}_p(f)=0$).
Let $\{\chi_0,\dots, \chi_{N-1}\}$ denote the set of characters of the cyclic group $C_N$. Embed this group as $\sigma(C_N)\subset \A(T)$, where $\sigma(r)z=z+\alpha.$ Let $\twp=\wp-\wp(\alpha)$.
Let $\pi_{\chi_j}$ denote the projection $\mathcal{O}_{\mathbb{T}}\rightarrow \mathcal{O}_{\mathbb{T}}^{\chi_j}$ given by $$\pi_{\chi_j}=\frac{1}{N}\sum_{r\in C_N}\overline{\chi_j(r)}\tilde{\sigma}(r).$$ We shall sometimes drop the notation that includes $\sigma$ and only write $\gamma f$ for a group action of $\Gamma$ on some $f\in \ot$. However, when this notation causes potential confusion, we opt for readability and write a dot. Consider the function $\wp'/\twp$. By \eqref{wpRel} and the fact that $\wp'$ is odd, we have $(\wp')=-3(0)+(1/2)+(\tau/2)+((1+\tau)/2)$. If we assume $N\geq 3$, then $\alpha$ cannot be a half lattice point and thus $\alpha\neq -\alpha$. The divisor of $\wp'/\twp$ is then given by \begin{align}\label{divisor}
(\wp'/\twp)=-(0)-(\alpha)-(-\alpha)+(1/2)+(\tau/2)+((1+\tau)/2),
\end{align} and thus we see that $\wp'/\twp\in \ot$. If $N=2$, then one can verify that $\wp'/\twp$ has poles in $\{0,\alpha\}$ and thus again $\wp'/\twp\in \ot$.

We define $P_j:T\rightarrow \CC$ to be the projection of $\wp'/\twp$ under $\pi_{\chi_j}$, multiplied by $N$: \begin{align}
P_j=N\pi_{\chi_j}\left(\frac{\wp'}{\twp}\right)=\sum_{k=0}^{N-1}\frac{r^k\wp'}{\omega_N^{kj}r^k\twp}.\end{align}
Notice that indeed $r\cdot P_j=\chi_j(r)P_j$. We would like to stress that $P_j$ depends both on a choice of a homomorphism $\sigma:C_N\rightarrow \A(T)$ and the character $\chi_j$ of $C_N$, while the notation only shows $j$ dependence. To avoid any confusion, we could write $P_{\sigma, \chi_j}$ instead - now risking an overload of notation however. \\

The next lemma states how $P_j$ transforms under the action of $s:z\mapsto -z$. It is a simple consequence of the fact that $\wp'/\twp$ is odd and that if $r$ is a translation of the torus, then $sr^k=r^{-k}s$.  \begin{Lemma}\label{action of s}
Consider the action on $T$ given by $s:z\mapsto -z$. Then $s$ acts on $P_j$ as $sP_j=-P_{-j}.$
\end{Lemma}
We will now list some basic properties of the functions $P_j$. Recall that if a function $f$ has a Laurent expansion $f(z)=\sum_{n\in \ZZ}c_n(z-z_0)^n$ about $z=z_0$, then the \emph{residue} of $f$ at $z=z_0$ is given by the coefficient $c_{-1}$. Notation: $\mathrm{Res}_{z=z_0}(f)=c_{-1}.$

\begin{Proposition}\label{prop:Expansion Vj}
Let $\alpha$ be an $N$-torsion point with $N\geq 2$, and let $\mathcal{S}=\{0,\alpha,2\alpha,\dots,(N-1)\alpha\}$. Then 
\begin{enumerate}
\item $P_j$ is a meromorphic function on $T$ for all $j$, with at most order 1 poles in $\mathcal{S}$. 
\item If $P_j$ has a pole, then it has poles everywhere in $\mathcal{S}$. 
\item The residue of $P_j$ about $z=0$ is given by
$$\mathrm{Res}_{z=0}(P_j)=-2+\omega_N^{-j}+\omega_N^{j}.$$
\end{enumerate} 
In particular $P_j$ is constant if and only if $j\equiv0 \mod N$. In fact, $P_{0}=0$.
\end{Proposition}
\begin{proof}
Let $v=\frac{\wp'}{\twp}$ and first take $N\geq 3$. Assume $r$ is given by $r:z\mapsto z+\alpha$. To prove the first item, note that $(r^k\tilde{\wp}) = -2(k\alpha)+ ((k+1)\alpha)+((k-1)\alpha)$ and hence $$(1/r^k\twp)=-((k-1)\alpha)-((k+1)\alpha)+2(k\alpha).$$ Using the divisor in \eqref{divisor}, we see that $r^kv=r^k\left(\frac{\wp'}{\twp}\right)$ has order 1 poles at $(k-1)\alpha$, $k\alpha$ and $(k+1)\alpha$, which are distinct when $N\geq 3$. Therefore $P_j=\sum_{k=0}^{N-1}\frac{r^k\wp'}{\omega_N^{kj}r^k\twp}$ has at most order 1 poles in $\mathcal{S}=\{0,\alpha,2\alpha,\dots,(N-1)\alpha\}$.\\ 

To prove the second statement, we observe that the divisor $(P_j)$ is invariant under $r$, since $r$ acts by a nonzero scalar on $P_j$. Hence if $P_j$ has a pole, then it has poles of order 1 everywhere in $\mathcal{S}$.\\

For the third item we consider a number of expansions about $z=0$.
The summands of $P_j$ that are responsible for the potential occurrence of pole at $z=0$ are $r^{-1}v, v$ and $rv$. It follows from the definition of $\wp$ in \eqref{DefWp} that about $z=0$, we have $$\frac{1}{\wp(z)-\wp(\alpha)}=z^2+O(z^4),\quad \wp'(z)=-\frac{2}{z^3}+O(z).$$ Thus \begin{align}\label{Expansion:v}
v(z)=\frac{\wp'}{\twp}(z)= -\frac{2}{z}+O(z).
\end{align} For $N\geq 3$ we have
$r\twp(z)=-\wp'(\alpha)z+O(z^2).$
Observe that $\wp'(\alpha)\neq 0$ as soon as $N\geq 3$, because zeros of $\wp'$ appear at half lattice points (i.e. 2-torsion points).
We see
$\frac{1}{r\twp}(z)= -\frac{1}{\wp'(\alpha)}\frac{1}{z}+O(1)$ and $r\twp'(z)=-\wp'(\alpha)+O(z)$ 
from which we infer $$rv(z)=\frac{1}{z}+O(1).$$ Similarly, one argues that $r^{-1}v(z)=\frac{1}{z}+O(1).$ Thus indeed $$\mathrm{Res}_{z=0}(P_j)=\mathrm{Res}_{z=0}\left(\sum_{k=0}^{N-1}\omega_N^{-kj}r^kv\right)=-2+\omega_N^{-j}+\omega_N^{j}.$$
If $N=2$, then $P_{1}=\frac{\wp'}{\twp}-r\frac{\wp'}{\twp}$ and it is easy to see that it has order 1 poles in $\{0,\alpha\}$. The expansion in \eqref{Expansion:v} still holds, thus $\frac{\wp'}{\twp}(z)=-\frac{2}{z}+O(z)$ and one can show that $r\frac{\wp'}{\twp}(z)=\frac{2}{z}+O(1)$ about $z=0$.\\

Now consider $P_j$ for $j\equiv 0\mod N$. Observe that it is invariant under $r$. Thus $P_{0}$ descends to a function $\tilde{P}_0$ on the torus $\tilde{T}=T/\langle r\rangle$. If $P_{0}$ has poles, they are of order 1 and if this is the case, then $\tilde{P}_0$ would have a simple pole on $\tilde{T}$. This is not possible due to the well known fact that if a compact Riemann surface has a meromorphic function with a single simple pole, it must be isomorphic to the Riemann sphere, cf. \cite[Proposition II.4.11]{miranda1995algebraic}. Hence $P_{0}$ is constant.

Finally, $sP_{0}=-P_{0}$ from which it follows that $P_{0}=0$. This concludes the proof. 
 \end{proof}
In Section \ref{MainResults}, we will introduce intertwining maps which will be used in the construction of bases of aLias. These maps are constructed via the $P_j$'s. As a consequence, the bases will be formulated in terms of the functions $P_j$. These functions are also important in another way: they generate the ring $\ot$ over the complex numbers, as we will prove now.
\begin{Proposition}\label{prop:Hauptmodul}
Let $r\in \A(T)$ be defined by $r(z)=z+\alpha$ with $\alpha$ an $N$-torsion point, where $N\geq 3$. Define $P_j$ with respect to this $r$. Then $$\ot=\CC[P_{1},P_{2}, \ldots, P_{N-1}],$$ where $\TT=T\setminus \langle r\rangle\cdot \{0\}.$ Furthermore, for any integer $j\not\equiv 0\mod N$, we have $$\CC[\wp_{\LL_{(\alpha)}}]=\CC[P_{-j}P_j].$$
Formulated differently, $P_{-j}P_j$ is a generator for the algebra of even meromorphic functions on the torus $T'=\CC/\LL_{(\alpha)}$, holomorphic on $T'\setminus\{0\}$. 
\end{Proposition}
\begin{proof}
 The space of meromorphic functions on a complex torus $T$ is denoted by $\mathcal{M}(T)$. Let $D$ be a divisor on $T$. The degree $\deg(D)$ of $D=\sum_{p\in T}n_p(p)$ is defined as $\deg(D)=\sum_{p\in T}n_p.$ Introduce the space of meromorphic functions on $T$ with poles bounded by $D$, denoted by $L(D)$: $$L(D)=\{f\in \mathcal{M}(T): \mathrm{div}(f)\geq -D\}.$$
Note that $L(D)$ is a complex vector space. The Riemann-Roch theorem, specialised to genus 1 Riemann surfaces \cite[Proposition V.3.14]{miranda1995algebraic}, says that if $D$ is a divisor on a complex torus $T$ and $\deg(D)>0$, then $\dim_{\CC}L(D)=\deg(D)$. Let $D=\sum_{j=0}^{N-1}(j\alpha).$ We see that $\dim_{\CC}L(D)=\deg(D)=N$. By construction, we have $P_j\in L(D)^{\chi_j}$, and hence $\dim_{\CC}L(D)^{\chi_j}= 1$ if $j\not\equiv 0\mod N$. Therefore, $L(D)=\CC\langle 1,P_{1},\ldots, P_{N-1}\rangle.$ We proceed inductively. For each $k\in \{1,\ldots, N-1\}$, going from $L((m-1)D)^{\chi_k}$ to $L(mD)^{\chi_k}$, we add the 1-dimensional subspace generated by a product of $P_j$'s such that $0\neq \prod_{j\in \{j_1,\ldots, j_m\}}P_j\in L(mD)^{\chi_k}$ to $L((m-1)D)^{\chi_k}$. In particular, for $k\not\equiv 0\mod N/2$ and letting $Q=(P_{-j}P_j)^{\lfloor \frac{m-1}{2}\rfloor}$ (where $\lfloor x\rfloor$ denotes the integer part of $x$), we claim that $$L(mD)^{\chi_k}/L((m-1)D)^{\chi_k}=\begin{cases} \CC P_kQ+L((m-1)D)^{\chi_k},\quad \text{if}\,\, m\,\, \text{odd}\\
 \CC P_{-k}P_{2k}Q+L((m-1)D)^{\chi_k},\quad \text{if}\,\, m\,\, \text{even}.
 \end{cases}$$ 
 Here we used that $N\geq 3$, so that there are enough integers $k$ such that $P_{2k}\neq 0$. For $k\equiv 0\mod N/2$ and $m$ even, we need to change $P_{-k}P_{2k}$ to $P_{-k-n}P_{2k+n}\in\ot^{\chi_k}$ for some suitable $n\in\ZZ$. This shows that $\ot=\bigcup_{j\in \NN}L(jD)$ is generated as a $\CC$-algebra by $P_{1},\ldots, P_{N-1}$ whenever $N\geq 3$.\\
To prove the second claim, notice that $r P_{-j}P_j=P_{-j}P_j$ and that it has poles of order 2 in $\LL_{(\alpha)}$ if and only if $j\not\equiv 0\mod N$. If $z_0$ is a zero of $P_{-j}$, then $-z_0$ is a zero of $P_j$ since $P_{-j}(z_0)=-P_j(-z_0)$ by Lemma \ref{action of s}. Hence the divisor of $P_{-j}P_j$ equals $(P_{-j}P_j)=-2(0)+(z_0)+(-z_0)$. Now recall that $\wp_{\LL_{(\alpha)}}$ has an order 2 pole at $z=0$ and is even, hence $(\wp_{\LL_{(\alpha)}})=-2(0)+(w_0)+(-w_0)$ for some $w_0\in T'$. We see that indeed $P_{-j}P_j=c_1\wp_{\LL_{(\alpha)}}+c_2$ for some $c_1,c_2\in \CC$ and $c_1\neq 0$ if and only if $j\not\equiv 0\mod N$. Therefore $\CC[\wp_{\LL_{(\alpha)}}]=\CC[P_{-j}P_j]$ if and only if $j\not\equiv 0\mod N$.
\end{proof}
The following identities related to the $P_j$'s are recorded in the next lemma. These will play an important role in the construction of intertwiners in Section \ref{MainResults}.
\begin{Lemma}\label{bNonZero}
Let $N\in \NN$. For every $j\in \ZZ$ and $k\not\equiv 0\mod N$, there are $\lambda,\mu\in \CC$ such that $$P_{2j}P_{-j}^2-P_{-2j}P_j^2=\lambda P_{-k}P_{k}+\mu.$$
If in addition $N\geq 3$, $k=\pm j$ and $j\not \equiv 0 \mod N/2$, then $\mu$ is nonzero.

\end{Lemma}
\begin{proof}
Let $j\in \ZZ$ and write $f=P_{2j}P_{-j}^2-P_{-2j}P_j^2$. Clearly $\gamma f=f$ for all $\gamma\in \langle s,r\rangle\cong D_N$ for any $N\in \NN$. In particular, $f$ is even and hence has at most order 2 poles, for any choice of $j$. By Proposition \ref{prop:Hauptmodul}, we see that for all $k\not\equiv 0\mod N$ we have that $f=\lambda P_{-k}P_{k}+\mu$ for some $\lambda,\mu\in \CC$. 

Assume now that $N\geq 3$, $k=\pm j$ and $j\not \equiv 0 \mod N/2$. 
If $\mu$ would be zero and $z_0\in T$ is a zero of $P_j$, then $f(z_0)=P_{2j}(z_0)P_{-j}^2(z_0)=0.$ 
Thus $z_0$ would also be a zero of $P_{-j}$ or $P_{2j}$. However, since the set of zeros of $P_{k}$ is invariant under $r$, the set of zeros of $P_j$ must then coincide with the zero set of $P_{-j}$ or $P_{2j}$. Since the order and location of the poles of $P_j$, $P_{-j}$ and $P_{2j}$ already coincide, this means that $(P_j)=(P_{-j})$ or $(P_j)=(P_{2j})$, that is, $P_j$ is either a constant multiple of $P_{-j}$ or $P_{2j}$. This is absurd, which therefore means that $\mu$ cannot be zero. 
\end{proof}
\begin{Remark}
One could find explicit formulas for $\lambda$ and $\mu$ using Proposition \ref{prop:Expansion Vj} and more detailed expansions. However, the obtained formulas are unlikely to benefit the transparency of the normal form we will later obtain. Related constructions will be carried out for $\Gamma=C_2\times C_2$ in Section \ref{subsectionC2xC2}, and here it does seem appropriate to keep track of the constants, mainly because there is simply one group instead of a collection of groups, and the constants have transparent formulas in terms of well known functions.
\end{Remark}

\section{Construction of normal forms}\label{MainResults}
In Section \ref{Setup}, we have seen a basic example of an aLia on a complex torus with symmetry group $C_2$, embedded inside $\A(T)$ as $z\mapsto -z$, and developed a number of functional aspects of the theory. We have observed that this $C_2$-aLia has a particularly transparent structure - it is generated over its algebra of invariants and its Cartan subalgebra has eigenvalues in $\CC$ (in fact, in this case it is a constant matrix). We will see that this is a general feature of aLias on complex tori with base Lie algebra $\mathfrak{sl}_2$. 

In this section we shall give explicit realisations of aLias and determine their isomorphism type. It turns out that if the symmetry group acts on a complex torus without fixed points, then we obtain aLias that are isomorphic to a current algebra. Formulated differently: they are untwisted. If on the other end the group acts with fixed points, the corresponding aLias are twisted. The twisted aLias turn out to be isomorphic to either the Onsager algebra $\mathfrak{O}$ or $\mf{S}_{\tau}$.\\
 
 We will bring the aLias into a \emph{normal form}, which will allow us to decide the isomorphism type of the aLias in question. The normal form will be defined in Definition \ref{def:NormalForm}. As a consequence, we settle the question whether the aLias are hereditary \cite{knibbeler2020hereditary}, as defined in Definition \ref{def:Hereditary}.
 
Our main tools in constructing bases for aLias, are certain $SL_2(\CC)$-valued maps on a complex torus $\TT=T\setminus \Gamma\cdot \{0\}$ whose composition with $\Ad:SL_2(\CC)\rightarrow \A(\mf{sl}_2)$ will intertwine the actions of $\Gamma$ on $T$ and $\mf{sl}_2$.
\begin{Definition}[Hereditary aLias\,\,\cite{knibbeler2020hereditary}]\label{def:Hereditary}
An aLia is said to be \emph{hereditary} if it contains a \emph{constant spectrum Cartan subalgebra (CSA)}. This is a Cartan subalgebra which is generated by an element $H$ with constant $\mathrm{ad}(H)$-eigenvalues.
\end{Definition}
\begin{Definition}[Normal form of aLias]\label{def:NormalForm}
We say that an aLia $\mathfrak{A}=(\mathfrak{sl}_2\otimes_{\CC}\mathcal{O}_{\mathbb{T}})^{\Gamma}$ is in \emph{normal form} if $$\mathfrak{A}=\CC\langle E,F,H \rangle \otimes_{\CC} \mathcal{O}_{\mathbb{T}}^{\Gamma}$$ such that $$[H,E]=2E,\quad [H,F]=-2F,\quad [E,F]=H\otimes p,$$
where $p\in \mathcal{O}_{\mathbb{T}}^{\Gamma}$.
\end{Definition}
We remark that for the normal form as in Definition \ref{def:NormalForm}, we automatically have that $H$ has eigenvalues in $\CC$. The CSA $\mf{h}=\CC H\otimes \mathcal{O}_{\mathbb{T}}^{\Gamma}$ is a constant spectrum CSA.\\

In the upcoming sections, we will construct normal forms for all aLias on complex tori with base Lie algebra $\mathfrak{sl}_2$, in our classification. This will therefore also establish the property of being hereditary of this class of aLias. \\

Let the symmetry group be given by $\Gamma=G\ltimes K\subset \A_0(T)\ltimes t(T)$, where we recall the notation $t(T)$ for the normal subgroup of translations of the torus $T$.
 We start with a lemma that says that, under certain circumstances, a normal form for the aLia with symmetry group $K$ automatically determines a normal form for $\Gamma$ as well. \begin{Lemma}\label{Nnormal-form}
Let $\Gamma=G\ltimes K\subset \A_0(T)\ltimes t(T)$. If $\mf{A}=(\mf{sl}_2\otimes_{\CC}\mathcal{O}_{\mathbb{T}})^{K}$ has a normal form, and $H\in \mf{A}$ is $G$-invariant, such that it generates a CSA of $\mf{A}$, then $(\mf{sl}_2\otimes_{\CC}\mathcal{O}_{\mathbb{T}})^{\Gamma}$ has a normal form as well.
\end{Lemma}
\begin{proof}
First recall that $G=C_{\ell}\subset \A_0(T)$ and that $K$ consists of translations of $T$. By Lemma \ref{lem:CN invariants}, we therefore have $\mathcal{O}_{\mathbb{T}}^K=\CC[\wp,\wp']$ for some $\wp$ associated to a suitable lattice. Observe that the statement is trivially true for $\ell=1$, hence we take $\ell\in \{2,3,4,6\}$.\\

 Assume the normal form of $(\mf{sl}_2\otimes_{\CC}\mathcal{O}_{\mathbb{T}})^{K}$ to be $$(\mf{sl}_2\otimes_{\CC}\mathcal{O}_{\mathbb{T}})^{K}=\CC\langle E,F,H\rangle \otimes_{\CC}\mathcal{O}_{\mathbb{T}}^K=\CC\langle E,F,H\rangle \otimes_{\CC}\CC[\wp,\wp']\cong \mf{sl}_2\otimes_{\CC}\mathcal{O}_{\mathbb{T}/K},$$ with $\ot^K$-linear brackets $[H,E]=2E$, $[H,F]=-2F$ and $[E,F]=H$. 
 Now, since $K$ is normal in $\Gamma$, $g\in G$ acts on $\CC\langle E,F,H\rangle \otimes_{\CC}\mathcal{O}_{\mathbb{T}}^K$ and $\{g\cdot E,g\cdot F,g\cdot H\}$ forms again a standard $\mf{sl}_2$-triple, where the action is defined in \eqref{DiagonalAction}. By assumption, $g\cdot H=H$. Furthermore, $[H,g\cdot E]=2g\cdot E$, so $g\cdot E=E\otimes k_1$ for some function $k_1\in \mathcal{O}_{\mathbb{T}}^K$. Similarly, $g\cdot F=F\otimes k_2$ for some $k_2\in \mathcal{O}_{\mathbb{T}}^K$.  We have $[g\cdot E,g\cdot F]=H$ and this implies that $k_2=k_1^{-1}$. The only units in $\mathcal{O}_{\mathbb{T}}^K=\CC[\wp,\wp']$ are the nonzero constants. Therefore $g\cdot E=\omega_{\ell}^jE$ and $g\cdot F=\omega_{\ell}^{-j}F$, for some $j\in \{1,\ldots, \ell-1\}$ such that $\gcd(j,\ell)=1$. Notice that this forces $j\in \{1,\ell-1\}$, due to the limited amount of choices for $\ell$. \\
 
 We now decompose $\CC\langle E,F,H\rangle \otimes_{\CC}\CC[\wp,\wp']$ into its $C_{\ell}$ isotypical components: $$(\CC\langle E,F,H\rangle \otimes_{\CC}\CC[\wp,\wp'])^{G}=\CC\langle E\otimes p_1,F\otimes p_2,H\rangle\otimes_{\CC} \CC[J],$$
where $p_1\in \CC[\wp,\wp']^{\chi_{\ell-j}}$, $p_2\in \CC[\wp,\wp']^{\chi_j}$ and $J$ generates $\mathcal{O}_{\mathbb{T}}^{\Gamma}$, which follows from Lemma \ref{IsotypicalComponents}. 

Its Lie structure is given by $$[H,E\otimes p_1]=2E\otimes p_1,\quad [H,F\otimes p_2]=-2F\otimes p_2,\quad [E\otimes p_1,F\otimes p_2]=H\otimes p_1p_2,$$
where $p_1p_2\in \CC[J]$.
We may assume that $g\cdot E=\omega_{\ell}E$ and $g\cdot F=\omega_{\ell}^{-1}F$ after applying an automorphism that interchanges the isotypical components if necessary (e.g. $\mathrm{Ad}\begin{pmatrix}
0 & 1\\
-1 &0
\end{pmatrix}$) and assume therefore that $p_1\in \CC[\wp,\wp']^{\chi_{\ell-1}}$, $p_2\in \CC[\wp,\wp']^{\chi_{1}}$. 

\end{proof}
Observe that Lemma \ref{Nnormal-form} implies that in order to understand aLias with nonabelian symmetry groups $\Gamma=G\ltimes K$, it is sufficient to understand the aLias with abelian symmetry groups $G$ and $K$ separately, granted that we have found a $K$-invariant element $H$ in a CSA, which is simultaneously $G$-invariant.

The proof of Lemma \ref{Nnormal-form} gives us a construction to obtain the normal for the aLia with symmetry group $\Gamma$ as well, if the condition on a generator of the CSA is being satisfied. \\

The next proposition gives a criterion for the existence of normal forms of aLias and extends Lemma \ref{Nnormal-form} by providing a normal form of the full group $\Gamma=G\ltimes K$.
\begin{Proposition}\label{Mauto}
Let $\Gamma=G\ltimes K\subset \A_0(T)\ltimes t(T)$ and let $\rho:\Gamma\rightarrow \A(\mf{sl}_2)$ be a monomorphism and  $\tilde{\rho}:\Gamma\rightarrow \A(\mf{sl}_2)$ a homomorphism whose kernel equals $K$. If there exists an inner automorphism $\Psi\in \A_{\mathcal{O}_{\mathbb{T}}}(\mf{sl}_2\otimes_{\CC}\mathcal{O}_{\mathbb{T}})$ such that for all $z\in \TT$ and $\gamma\in \Gamma$, we have 
\begin{align}\label{Mintertwiner}
\Psi(\gamma\cdot z)=\rho(\gamma)\Psi(z)\tilde{\rho}(\gamma)^{-1},
\end{align}
then

\begin{enumerate}
\item $\mf{K}:=(\mf{sl}_2\otimes_{\CC}\mathcal{O}_{\mathbb{T}})^{K}\cong \mf{sl}_2\otimes_{\CC} \mathcal{O}_{\mathbb{T}}^{K},$
\item  There exists a normal form for $\mathfrak{A}=(\mf{sl}_2\otimes_{\CC}\mathcal{O}_{\mathbb{T}})^{\Gamma}$ and if the normal form of $\mf{K}$ is given by $\CC\langle E,F,H \rangle \otimes_{\CC} \mathcal{O}_{\mathbb{T}}^{K}$, with its Lie structure inherited from $\mf{sl}_2$, then the normal form for $\mathfrak{A}$ given by $$(\mf{sl}_2\otimes_{\CC}\mathcal{O}_{\mathbb{T}})^{\Gamma}=\CC\langle E\otimes p_1,F\otimes p_2,H \rangle \otimes_{\CC} \mathcal{O}_{\mathbb{T}}^{\Gamma},$$ with Lie structure $$[H,E\otimes p_1]=2E\otimes p_1,\quad [H,F\otimes p_2]=-2F\otimes p_2,\quad [E\otimes p_1,F\otimes p_2]=H\otimes p_1p_2,$$
where $p_1$ generates $\mathcal{O}_{\mathbb{T}}^{\chi_{\ell-1}}$, where $|G|=\ell$, $p_2$ generates $\mathcal{O}_{\mathbb{T}}^{\chi_{1}}$ and $p_1p_2\in \mathcal{O}_{\mathbb{T}}^{\Gamma}$.
\end{enumerate}
\end{Proposition}
\begin{proof}
First of all, we would like to stress that we can view $\Psi\in \A_{\mathcal{O}_{\mathbb{T}}}(\mf{sl}_2\otimes_{\CC}\mathcal{O}_{\mathbb{T}})$ as a holomorphic $\A(\mf{sl}_2)$-valued function on $\mathbb{T}$, in which case we use the notation $\Psi(z)$. We shall use both interpretations interchangeably and let the argument of $\Psi$ be either an element of $\mathbb{T}$ or an element of $\mf{sl}_2\otimes_{\CC}\ot$.\\

The first part claims that for $k\in K$, if we have $\Psi(k\cdot z)=\rho(k)\Psi(z)$, then there is a normal form for $\mf{K}$. Let $\{x_1,x_2,x_3\}$ be a basis for $\mf{sl}_2$. Then  for $j=1,2,3,$ we have that $X_j(z):=\Psi(z)x_j$ is $K$-invariant: $$k\cdot X_j(z)=k\cdot \Psi(z)x_j=\rho(k)\Psi(k^{-1}\cdot z)x_j=X_j(z).$$ Since $\Psi\in \A_{\mathcal{O}_{\mathbb{T}}}(\mf{sl}_2\otimes_{\CC}\mathcal{O}_{\mathbb{T}})$ is an inner automorphism, we have $$\CC\langle x_1,x_2,x_3\rangle\otimes_{\CC}\ot=\Psi(\CC\langle x_1,x_2,x_3\rangle\otimes_{\CC}\ot)=\CC\langle X_1,X_2,X_3\rangle\otimes_{\CC}\ot.$$ Taking the invariants, we obtain a normal form $$\mf{K}=\CC\langle X_1,X_2,X_3\rangle\otimes_{\CC}\ot^K\cong \mf{sl}_2\otimes_{\CC}\CC[\wp,\wp'],$$ with the brackets inherited from $\mf{sl}_2$ and where $\wp,\wp'$ generate $\ot^K$.\\

To prove the second item, we will argue as follows. Let $\gamma=gk\in \Gamma=G\ltimes K$ and $x\in \mf{sl}_2$. By \eqref{Mintertwiner}, $$\gamma\cdot \Psi(z)x=\Psi(z)\tilde{\rho}(\gamma)x.$$ Since $G$ is cyclic, it leaves a CSA of $\mf{sl}_2$ invariant. Hence there is a basis $\{X_1,X_2,X_3\}\subset \mf{sl}_2$, with $X_1$ a generator in the CSA, such that with respect to this basis, $\tilde{\rho}(\gamma)=\tilde{\rho}(gk)=\diag(1,\chi_1(g),\chi_{\ell-1}(g))$, where $\chi_j$ are the characters of $G$, $g$ generates $G$ and $k\in K$. Now, $\Psi(X_1\otimes 1)$ generates a CSA of $\mf{K}$ and $g\cdot \Psi(X_1\otimes 1)=\Psi(X_1\otimes 1)$ for all $g\in G$ by \eqref{Mintertwiner}. We can therefore apply Lemma \ref{Nnormal-form} to arrive at the conclusion of the existence of a normal form for $\mf{A}$ and the stated form of it.
\end{proof}

Recall that the groups $\Gamma=G\ltimes K\subset \A_0(T)\ltimes t(T)$ considered in our classification, are isomorphic to one of the following groups: $$1\ltimes C_N, \quad C_2\ltimes C_N,\quad C_{3}\ltimes (C_2\times C_2),\quad C_{\ell}\ltimes 1,$$
where $\ell\in \{1,2,3,4,6\}$ and $N\in \NN$.
The main object of our considerations for the rest of the paper is finding a suitable $\Psi\in \A_{\mathcal{O}_{\mathbb{T}}}(\mf{sl}_2\otimes_{\CC}\mathcal{O}_{\mathbb{T}})$ such that \eqref{Mintertwiner} holds. Section \ref{subsectionCNDN} is devoted to the cases $\Gamma=C_N$ and $\Gamma=D_N\cong C_2\ltimes C_N$ while Section \ref{subsectionC2xC2} concerns $\Gamma=C_2\times C_2$ and $\Gamma=A_4\cong C_3\ltimes (C_2\times C_2)$. 
\subsection{The cases $\Gamma=C_N$ and $\Gamma=D_N$}\label{subsectionCNDN}
We will start with the case $\Gamma=C_N$ where $C_N$ is embedded into $\A(T)$ as translations by the homomorphism $\sigma_{\alpha}:C_N\rightarrow \A(T)$, defined by $\sigma_{\alpha}(r)z=z+\alpha$ for some $N$-torsion point $\alpha\in T$. The main tool we use in the construction of aLias with symmetry group $C_N$, is the construction of map $\Phi_j:\TT\rightarrow SL_2(\CC)$ ($j\not\equiv0\mod N/2)$ for which the $\mathcal{O}_{\mathbb{T}}$-linear automorphism $\mathrm{Ad}(\Phi_j)$ will map a basis of $\mathfrak{sl}_2$ to a subalgebra of the current algebra $\mf{sl}_2\otimes_{\CC}\mathcal{O}_{\mathbb{T}}$. The automorphism $\mathrm{Ad}(\Phi_j)$ is designed such that it satisfies the equivariance condition \eqref{DefIntertwiner} for a given homomorphism $\rho:C_N\rightarrow \A(\mathfrak{sl}_2)$, which will imply that we obtain a basis of invariant matrices (with respect to the action induced by the homomorphism $\rho\otimes \tilde{\sigma})$. From here, one quickly obtains normal forms for aLias with symmetry group $C_N$.
Having established the relevant results, one easily obtains aLias with symmetry group $D_N$, since $D_N\cong C_2\ltimes C_N$. Due to Lemma \ref{lem:Klein} and \ref{lem:IsomALias}, we can choose any faithful action of $C_N$ on $\mathfrak{sl}_2$ without loss of generality. \\

A crucial step in our construction of normal forms relies on the existence of intertwining operators between the actions of a symmetry group $\Gamma$ on $\mathfrak{sl}_2$ and $T$. Let us briefly mention some general facts about these intertwiners to provide some motivation on the choice of these maps. Consider monomorphisms (that is, faithful homomorphisms) $\sigma:\Gamma\rightarrow \A(T)$ and $\delta:\Gamma\rightarrow SL_2(\CC)$.
We would like to construct a matrix-valued map $\Phi:T\rightarrow SL_2(\CC)$, 
\begin{align}\label{Pmatrix}
\Phi(z)=\begin{pmatrix}
p_1(z)& q_1(z)\\
p_2(z) & q_2(z)
\end{pmatrix},
\end{align} where $p_j,q_j\in \mathcal{O}_{\mathbb{T}}$, such that $\Phi(\sigma(\gamma)z)=\delta(\gamma)\Phi(z)$ for all $\gamma\in \Gamma$ and $z\in \mathbb{T}$. Taking the adjoint $\mathrm{Ad}$ of both sides, gives us the equivariance property with respect to the actions induced by $\sigma: \Gamma\rightarrow \A(T)$ and $\rho:\Gamma\rightarrow \A(\mf{sl}_2)$, with $\rho=\Ad\circ\delta$:
\begin{align}\label{DefIntertwiner}
\mathrm{Ad}(\Phi(\sigma(\gamma)z))=\rho(\gamma)\mathrm{Ad}(\Phi(z)).
\end{align}
Notice that $\Ad(\Phi)$ is an element of $\A_{\ot}(\mf{sl}_2\otimes_{\CC}\ot)$ - the group of $\ot$-linear automorphisms of $\mf{sl}_2\otimes_{\CC}\ot$.
The difficult part in constructing such an equivariant map, is designing it in such a way that it is invertible over $\CC$. We will present a procedure for the group $C_N$ below. The case of $C_2\times C_2$ will be dealt with in Section \ref{subsectionC2xC2}. As we will see later, it will be enough to consider these two cases in the construction of normal forms.\\

Take an element $\gamma\in \Gamma$ of order $N$ and let $\delta:\Gamma\rightarrow SL_2(\CC)$ and $\sigma:\Gamma\rightarrow \A(T)$ be monomorphisms. 
We call a vector $v=(v_1,v_2)^T$ consisting of functions $v_j\in \mathcal{O}_{\mathbb{T}}$ an \emph{equivariant vector} with respect to the actions defined by $\delta$ and $\sigma$, if $$v(\sigma(\gamma)z)=\delta(\gamma)v(z),\quad \text{for all}\,\,\gamma\in \Gamma,\,\,z\in \mathbb{T}.$$ The columns of $\Phi$ consist of equivariant vectors $p:=(p_1,p_2)^T$ and $q:=(q_1,q_2)^T$ as in \eqref{Pmatrix}. We will now fix a choice of $\delta$, without loss of generality as we will see later. Let $\delta=\delta_j:\Gamma\rightarrow SL_2(\CC)$ be the monomorphism defined by $\delta_j(\gamma)=\diag(\omega_N^j,\omega_N^{-j})$ for some $j$ coprime to $N$. The condition $\Phi(\sigma(\gamma)z)=\delta_j(\gamma)\Phi(z)$ implies that $p_1,q_1\in \mathcal{O}_{\mathbb{T}}^{\chi_{N-j}}$, and $p_2,q_2\in \mathcal{O}_{\mathbb{T}}^{\chi_{j}}$. \\

We will now present a construction for the matrix $\Phi$ as defined in \eqref{Pmatrix}, that will be used in the construction of normal forms of both $C_N$ and $D_N$. 

 Let $N\geq 3$ and fix some integer $j\not\equiv 0\mod N/2$. Take $\lambda,\mu\in \CC$ such that $P_{2j}P_{-j}^2-P_{-2j}P_j^2=\lambda P_{-j}P_j+\mu.$ Lemma \ref{bNonZero} guarantees that $\mu$ is nonzero. 
Let the columns of $\Phi$ be given by $p$ and $q$. In a sense, the simplest $C_N$-equivariant vector we can choose for $p$ is $(P_{-j}, P_j)^T$. 
Given this choice for $p$, for $q$  we need to find $q_1\in \mathcal{O}_{\mathbb{T}}^{\chi_{N-j}}$ and $q_2\in \mathcal{O}_{\mathbb{T}}^{\chi_{j}}$ such that $\det(\Phi)=P_{-j}q_2-P_jq_1=1$. Let $s$ be given by $s(z)=-z$ and recall that $s\cdot P_j=-P_{-j}$ (cf. Lemma \ref{action of s}). One finds $$p(s\cdot z)=-\begin{pmatrix}
0& 1\\
1 &0
\end{pmatrix} p(z).$$ 
Let $S=\begin{pmatrix}
0& 1\\
1 &0
\end{pmatrix}$. 
In light of the equivariance condition \eqref{Mintertwiner}, we will impose the condition $q(s\cdot z)=Sq(z),$ from which we get 
\begin{align}\label{DN-equivariance}
\Phi(s\cdot z)=S\Phi(z)\begin{pmatrix}
-1& 0\\
0 & 1
\end{pmatrix}.
\end{align}

Adding this condition to the equation $P_{-j}q_2-P_jq_1=1$, yields a unique solution in the space $L(2D)\subset \mathcal{O}_{\mathbb{T}}$, where $D$ is the divisor $D=\Gamma\cdot(0)=\sum_{j=0}^{N-1}(j\alpha)$. 
The solution to this equation in $L(2D)$ is provided by Lemma \ref{bNonZero}. By the proof of Proposition \ref{prop:Hauptmodul}, we deduce that for each integer $j\not\equiv 0\mod N/2$, we have the vector space decomposition $L(2D)^{\chi_j}=\CC P_j\oplus \CC P_{-j}P_{2j}$. Then, if we choose $$q_1=\frac{1}{\mu}P_jP_{-2j}+\frac{\lambda}{2\mu} P_{-j}\in \mathcal{O}_{\mathbb{T}}^{\chi_{N-j}},\quad q_2=\frac{1}{\mu}P_{-j}P_{2j}-\frac{\lambda}{2\mu} P_j\in \mathcal{O}_{\mathbb{T}}^{\chi_{j}},$$ we satisfy the requirement by Lemma \ref{bNonZero}, and this solution is unique by considering the vector space decompositions of $L(2D)^{\chi_j}$ and $L(2D)^{\chi_{N-j}}$. Thus the matrix $\Phi$ with the $p_j$ and $q_j$ defined above, is the unique matrix with lowest order poles satisfying the intertwining condition.
\\

We can summarise this as follows: given a complex torus $T$, define a faithful homomorphism $\sigma:C_N\rightarrow t(T)$. Set $\TT=T\setminus C_N\cdot \{0\}$ and consider the functions $P_j$ defined with respect to the choice of $\sigma$.
For every integer $j\not\equiv 0\mod N/2$, we associate to $\delta_j$ and $\sigma$ a meromorphic, $C_N$-equivariant map $\Phi_j:T\rightarrow SL_2(\CC)$, holomorphic on $\TT$, given by 
\begin{align}\label{intertwiner1}
\Phi_j(z)=\begin{pmatrix}
P_{-j}(z) & \frac{1}{\mu}P_j(z)P_{-2j}(z)+\frac{\lambda}{2\mu} P_{-j}(z)\\
P_j(z) &\frac{1}{\mu}P_{-j}(z)P_{2j}(z)-\frac{\lambda}{2\mu} P_j(z)
\end{pmatrix},
\end{align} 
where $\lambda$ and $\mu$ are constants defined in Lemma \ref{bNonZero}.
We will record our discussion below in two lemmata.
\begin{Lemma}\label{LemmaDeterminant}
For any integer $j\not\equiv 0\mod N/2$  
 it holds that $\det(\Phi_j)=1$. 
\end{Lemma}
\begin{Lemma}\label{IntertwinerProperty}
Let $\sigma:C_N\rightarrow \A(T)$ be a monomorphism and $\delta_{j,N}:C_N\rightarrow SL_2(\CC)$ be defined by $\delta_{j,N}(r)=\diag(\omega_N^j,\omega_N^{-j})$. For all integers $j\not\equiv 0\mod N/2$, $z\in \mathbb{T}$ and $r\in C_N$, we have 
\begin{align}\label{eq:Intertwiner}
\Phi_j(\sigma(r) z)=\delta_{j,N}(r)\Phi_j(z).
\end{align} 
\end{Lemma} 
Define $\rho_{j,N}=\Ad\circ \delta_{j,N}$. Taking the adjoint of both sides of $\Phi_j(\sigma(r) z)=\delta_{j,N}(r)\Phi_j(z)$, we obtain $\Ad(\Phi_j(\sigma(r) z))=\rho_{j,N}(r)\Ad(\Phi_j(z))$. Notice that $\ker\rho_{j,N}\cong C_2$ if $N$ is even and $\gcd(j,N)=1$. 

We have that for any integer $j\not\equiv 0\mod N/2$, $\mathrm{Ad}(\Phi_j)\in \A_{\mathcal{O}_{\mathbb{T}}}(\mf{sl}_2\otimes_{\CC}\mathcal{O}_{\mathbb{T}})$ by Lemma \ref{LemmaDeterminant}. We can also consider $\mathrm{Ad}(\Phi_j)$ to be a meromorphic map from $T$ to $\A(\mf{sl}_2)$, $z\mapsto \mathrm{Ad}(\Phi_j(z))$ holomorphic on $\TT$. The map $\mathrm{Ad}(\Phi_j)$ intertwines the actions defined by $\sigma$ and $\rho$, or stated differently, it is equivariant with respect the action of $C_N$. 

\begin{Remark}\label{Rem:EqRep}
The matrix $\Phi(z):=\Phi_j(z)$ is constructed via a particular choice of $\rho$. If $\rho'$ is an equivalent representation of $C_N$, $\rho'(r)=B\rho(r) B^{-1}$ for all $r\in C_N$ for some $B\in \A(\mf{sl}_2)$, and we define $\Phi'=B\Phi B^{-1}$, then one can verify that $\rho'(r)\mathrm{Ad}(\Phi'(z))=\mathrm{Ad}(\Phi'(\sigma(r)z))$.
\end{Remark}
The next two theorems describe the aLias with a cyclic symmetry group. We split the results according to the genus $g$ of $T/\sigma(C_N)$. First, we consider $g(T/\sigma(C_N))=1$. 

\begin{Theorem}[Genus 1 case]\label{CNaLia}
Let $\rho:C_N\rightarrow \A(\mf{sl}_2)$ and $\sigma:C_N\rightarrow t(T)$ be monomorphisms. Then $$(\mathfrak{sl}_2\otimes_{\CC}\mathcal{O}_{\mathbb{T}})^{\rho\otimes \tilde{\sigma}(C_N)}\cong \mathfrak{sl}_2\otimes_{\CC}\ot^{\tilde{\sigma}(C_N)}.$$ A normal form is given by $$(\mathfrak{sl}_2\otimes_{\CC}\mathcal{O}_{\mathbb{T}})^{\rho\otimes \tilde{\sigma}(C_N)}=\CC\langle E,F,H\rangle\otimes_{\CC}\CC[\wp_{\LL},\wp_{\LL}'],$$ where $\LL$ is a suitable lattice, and with brackets $$[H,E]=2E,\quad [H,F]=-2F,\quad [E,F]=H,$$
where $E,F$ and $H$ are the images under $\Ad(\Phi_j)$, for some integer $j\not\equiv 0\mod N/2$, of $e\otimes 1$, $f\otimes 1$ and $h\otimes 1$, respectively.
\end{Theorem}
\begin{proof}
The main ingredients are Proposition \ref{Mauto} and Lemma \ref{IntertwinerProperty}. Fix a lattice $\LL=\ZZ\oplus \ZZ\tau$ and let $T=\CC/\LL$. Let $\alpha'\in \CC$ be such that the projection of $\alpha'$ under $\pi_{\LL}:\CC\rightarrow \CC/\LL$, $\alpha:=\pi(\alpha')$ is an $N$-torsion point of $T$. Thus $\alpha$ satisfies $N\alpha=0 \mod \LL$. We may assume that $\sigma$ is given by $\sigma_{\alpha}(r)z=z+\alpha$.

First, let $N$ be odd.  We may also assume that $\delta(r)=\delta_{j,N}(r):=\diag(\omega_N^j,\omega_N^{-j})$ for some $j$ with $\gcd(j,N)\neq 1$ and that $\rho$ is given by $\Ad\circ\delta$. The explanation is that if $\delta'$ is another monomorphism, and $\rho'=\Ad\circ\delta'$, then we know by Lemma \ref{lem:Klein} that $\rho(C_N)$ and $\rho'(C_N)$ are conjugated subgroups. Lemma \ref{lem:IsomALias} then tells us that the aLias defined by $\rho$ and $\rho'$ are isomorphic. 

The algebra of invariants $\ot^{\tilde{\sigma}_{\alpha}(C_N)}$ is given by $\CC[\wp_{\LL_{(\alpha)}},\wp_{\LL_{(\alpha)}}']$ by Lemma \ref{lem:CN invariants}.
Now, $\Ad(\Phi_j)$ satisfies \eqref{Mintertwiner} because of Lemma \ref{LemmaDeterminant} and Lemma \ref{IntertwinerProperty}, for which reason we can apply Proposition \ref{Mauto}. We therefore have, as in the proof of Proposition \ref{Mauto}, that the following elements form a basis of $\mathfrak{A}=(\mathfrak{sl}_2\otimes_{\CC}\mathcal{O}_{\mathbb{T}})^{C_N}$ over $\CC[\wp_{\Lambda_{(\alpha)}},\wp_{\Lambda_{(\alpha)}}']$:
$$E_j=\mathrm{Ad}(\Phi_j) e\otimes 1,\quad F_j=\mathrm{Ad}(\Phi_j) f\otimes 1,\quad H_j=\mathrm{Ad}(\Phi_j) h\otimes 1,$$
where $\Phi_j$ is given in \eqref{intertwiner1}. Remark \ref{Rem:EqRep} can be used to find a normal form if instead of $\delta$, we had chosen $\delta'$ as defined in the beginning of the proof.\\

Assume now $N$ is even. We may assume that $\rho$ is given by $$\rho(r)\begin{pmatrix}a&b\\c&-a
\end{pmatrix}=\begin{pmatrix}a&\omega_Nb\\\omega_N^{-1}c&-a
\end{pmatrix}.$$Observe that the image of $\rho$ and $\hat{\rho}=\Ad\circ \delta_{j,2N}$ coincide. Notice that $\ker\hat{\rho}=\langle r^{N}\rangle\cong C_2$, where $r\in C_{2N}$. 
Define $$\hat{\LL}=\begin{cases}
\ZZ2\oplus \ZZ\tau &\,\,\text{if}\,\, N\alpha/2\in \{1/2 \mod \LL, (1+\tau)/2\mod \LL\},\\
\ZZ\oplus \ZZ2\tau &\,\,\text{if}\,\, N\alpha/2=\tau/2\mod \LL.
\end{cases}$$ Let $\hat{T}=\CC/\hat{\LL}$ and define $\hat{\TT}$ accordingly. Define $\sigma_1:C_{2N}\rightarrow \A(\hat{T})$ by $\sigma_1(r)z=z+\pi_{\hat{\LL}}(\alpha')$ and notice that $\pi_{\hat{\LL}}(N\alpha')=N\alpha$ is a half lattice point of $\hat{\LL}$.\\

Recall that in the proof of Lemma \ref{lem:CN invariants} we used that for any $\alpha\in T$, $\mathcal{O}^{\tilde{\sigma}_\alpha(C_N)}_{\mathbb{T}}$ is the space of $\Lambda_{(\alpha)}$-periodic meromorphic functions which are holomorphic on $T\setminus \Lambda_{(\alpha)}$. This allowed us to say that $\ot^{\Gamma}=\mathcal{O}_{\TT/\Gamma}$. We adopt this perspective again.
Suppose $N\alpha/2=1/2\mod \LL$. Then $\hat{T}/\sigma_1(C_{2N})= \CC/(\ZZ2+\ZZ\tau+\ZZ\alpha')$. Notice that $N\alpha'=1+2n_1+2n_2\tau$ for some $n_1,n_2\in \ZZ$. We see that $1\in \ZZ2+\ZZ\tau+\ZZ\alpha'$ and hence $$\hat{T}/\sigma_1(C_{2N})= \CC/(\ZZ 2+\ZZ\tau+\ZZ\alpha')=\CC/(\ZZ+\ZZ\tau+\ZZ\alpha').$$
Also observe that $T/\sigma(C_N)= \CC/(\ZZ+\ZZ\tau+\ZZ\alpha')$. Therefore $\hat{T}/\sigma_1(C_{2N})=T/\sigma(C_N)$. Similarly we have $\hat{\TT}/\langle r^N\rangle=\TT$. The same claim holds in the cases where $N\alpha/2\in \{\tau/2,(1+\tau)/2\}$. \\

Define $\Phi:=\Phi_j$ associated to $\delta_{j,2N}$ and $\sigma_1$, for some integer $j\not\equiv N/2$. It is clear that $\Ad(\Phi)\in \A_{\mathcal{O}_{\hat{\TT}}}(\mf{sl}_2\otimes_{\CC}\mathcal{O}_{\hat{\TT}})$. By similar reasoning as in Proposition \ref{Mauto}, we have that $\Ad(\Phi)$ establishes an isomorphism 
$$\mf{B}:=(\mf{sl}_2\otimes_{\CC}\mathcal{O}_{\hat{\TT}})^{\hat{\rho}\otimes\tilde{\sigma}_1(C_{2N})}\cong \mf{sl}_2\otimes_{\CC}\mathcal{O}_{\hat{\TT}}^{\tilde{\sigma}_1(C_{2N})}.$$
We know $\hat{T}/\sigma_1(C_{2N})= T/\sigma(C_N)$ and therefore $\mathcal{O}_{\hat{\TT}}^{\tilde{\sigma}_1(C_{2N})}= \mathcal{O}_{\TT}^{\tilde{\sigma}(C_{N})}.$ Hence $\mf{B}\cong  \mf{sl}_2\otimes_{\CC}\mathcal{O}_{\TT}^{\tilde{\sigma}(C_{N})}.$
We claim that $\mf{B}$ is precisely $(\mathfrak{sl}_2\otimes_{\CC}\mathcal{O}_{\mathbb{T}})^{\rho\otimes\tilde{\sigma}(C_{N})}.$ By Lemma \ref{aLiaIdentity1} we have that $\mf{B}=(\mathfrak{sl}_2\otimes_{\CC}\mathcal{O}_{\hat{\mathbb{T}}/\langle r^N\rangle})^{\hat{\Gamma}}$, where $\hat{\Gamma}=\langle r\rangle/\langle r^N\rangle \cong C_N$. Now, since $\hat{\TT}/\langle r^N\rangle=\TT$ and the action of $\hat{\Gamma}$ is the same as the one induced from $\sigma$, we conclude that $\mf{B}=(\mathfrak{sl}_2\otimes_{\CC}\mathcal{O}_{\mathbb{T}})^{\rho\otimes\tilde{\sigma}(C_{N})}.$\\

We know $\Ad(\Phi(r^N\cdot z))=\Ad(\Phi(z))$, since $r^N\in \ker(\hat{\rho})$. Therefore, any invariant element $X:=\Ad(\Phi)x$, with $x\in \mf{sl}_2$, consists of entries which are invariant under $\langle r^N\rangle\cong C_2$. The remainder of the proof for $N$ is even, follows as it did for $N$ odd. A normal form is given by $$E=\mathrm{Ad}(\Phi) e\otimes 1,\quad F=\mathrm{Ad}(\Phi) f\otimes 1,\quad H=\mathrm{Ad}(\Phi) h\otimes 1.$$ This proves the claim.
\end{proof}
Explicitly, the generators $H_j:=\mathrm{Ad}(\Phi_j)h\otimes 1$,  $E_j:=\mathrm{Ad}(\Phi_j)e\otimes 1$ and $F_j:=\mathrm{Ad}(\Phi_j)f\otimes 1$ ($j\not\equiv 0 \mod N/2$) are given by \begin{align}
H_j = \frac{1}{\mu}\begin{pmatrix}
P_{-j}^2 P_{2j} + P_{-2j} P_j^2 & -2 P_{-j} P_j P_{-2j} -\lambda P_{-j}^{2} \\
2 P_{-j} P_j P_{2j} - \lambda P_j^{2}  & -P_{-j}^2 P_{2j} - P_{-2j} P_j^2
\end{pmatrix},
\quad E_j=\begin{pmatrix}
-P_{-j} P_j& P_{-j}^{2} \\
-P_j^{2}& P_{-j} P_j
\end{pmatrix}, \end{align}
and 
\begin{align}
F_j=\frac{1}{4\mu^2}\begin{pmatrix}
4P_{-j} P_j P_{-2j} P_{2j} + \lambda^{2}P_{-j} P_j  + 2\lambda \mu & -4P_j^{2} P_{-2j}^{2} - 4 \lambda P_{-j} P_j P_{-2j} -  \lambda^{2}P_{-j}^{2} \\
4P_{-j}^2 P_{2j}^{2} - 4 \lambda P_{-j} P_j P_{2j} +  \lambda^{2}P_j^{2} & -4P_j P_{-j} P_{2j} P_{-2j} - \lambda^{2} P_{-j} P_j- 2\lambda \mu\end{pmatrix}.
\end{align}
In the next corollary, we shall omit the $j$-dependence of $E,F,H$ and assume this implicitly.
\begin{Corollary}\label{DNaLia}
 Let $\rho:D_N\rightarrow \A(\mf{sl}_2)$ and $\sigma:D_N\rightarrow \A(T)$ be monomorphisms. Then $$(\mathfrak{sl}_2\otimes_{\CC}\mathcal{O}_{\mathbb{T}})^{\rho\otimes \tilde{\sigma}(D_N)}\cong \mf{S}_{\tau},$$ for some $\tau\in \mathbb{H}$.
A normal form is given by $$(\mathfrak{sl}_2\otimes_{\CC}\mathcal{O}_{\mathbb{T}})^{\rho\otimes \tilde{\sigma}(D_N)}=\CC\langle \tilde{E}, \tilde{F}, \tilde{H}\rangle\otimes_{\CC} \CC[\wp_{\Lambda}],$$ where $\LL$ is a suitable lattice and $
\tilde{E}=E\otimes \wp_{\Lambda}'$, $\tilde{F}=F\otimes \wp_{\Lambda}'$ and $\tilde{H}=H$, where $E,F,H$ are defined in the proof of Theorem \ref{CNaLia}.
The Lie structure is given by \begin{align}\label{DN brackets}
[\tilde{H},\tilde{E}]=2\tilde{E},\quad [\tilde{H},\tilde{F}]=-2\tilde{F},\quad [\tilde{E},\tilde{F}]=\tilde{H}\otimes (4\wp_{\Lambda}^3-g_2\wp_{\Lambda}-g_3).\end{align}
\end{Corollary}
\begin{proof}
We may assume that $\sigma$ is given by $\sigma(s)z=-z$ and $\sigma(r)z=z+\alpha$ for some $N$-torsion point $\alpha\in T$. Let $S=\begin{pmatrix}
0& 1\\
1 & 0
\end{pmatrix}$. 
We will use Proposition \ref{Mauto}. Recall $H=\mathrm{Ad}(\Phi)h,$ where $\Phi=\Phi_j$ for some $j\not\equiv 0\mod N/2$. To satisfy condition \eqref{Mintertwiner}, we only need to verify condition \eqref{DN-equivariance}: $$\Phi(s\cdot z)=S\Phi(z)\begin{pmatrix}
-1& 0\\
0 & 1
\end{pmatrix},$$
where $s$ generates $C_2\subset\A_0(T)$. This holds as a result of the construction of $\Phi$ in the beginning of this section. 
Thus $H$ is a $D_N$-invariant as well. Hence we can apply Lemma \ref{Nnormal-form} and use Example \ref{C2aLia} to obtain the stated normal form. 
\end{proof}
Notice that above corollary also covers the aLia $(\mathfrak{sl}_2\otimes_{\CC}\mathcal{O}_{\mathbb{T}})^{C_2}$, where  $g(T/C_2)=0$, since $D_1=C_2\ltimes 1\cong C_2.$\\

So far, we have considered the groups $C_N$ and $D_N\cong C_2\ltimes C_N$ for which $C_N$ is embedded as translations in $\A(T)$. There are a number of cyclic groups of small order which also allow an embedding as a rotation in $\A(T)$ for a torus $T$ with more symmetry than the generic case, cf. Proposition \ref{lem:groups up to conjugation}.
The groups $C_{\ell}$ for $\ell\in \{3,4,6\}$ can be embedded inside $\A(T)$ as rotations, for tori either isomorphic to $T_i$ or $T_{\omega_6}$. The following theorem describes the aLias with these symmetry groups.
\begin{Theorem}[Genus 0 case]\label{ClaLias}
Let $\rho:C_{\ell}\rightarrow \A(\mathfrak{sl}_2)$ and $\sigma:C_{\ell}\rightarrow \A(T)$ be monomorphisms and assume that $g(T/\sigma(C_{\ell}))=0$. Then $$(\mathfrak{sl}_2\otimes_{\CC}\mathcal{O}_{\mathbb{T}})^{\rho\otimes \tilde{\sigma}(C_{\ell})}\cong \mf{O},$$ if and only if $\ell\in \{3,4,6\}.$

 \end{Theorem}
\begin{proof}
By Lemma \ref{lem:IsomALias}, we only have to prove the statement for a particular choice of monomorphisms.
Let $\ell=3$ and define the homomorphism $\rho: C_{3}\rightarrow \A(\mf{sl}_2)$  by
$$
\rho(s)\begin{pmatrix}
a & b\\
c &-a
\end{pmatrix}=\begin{pmatrix}
a & \omega_3b\\
\omega_{3}^{-1}c &-a
\end{pmatrix}.
$$
Let $s\in C_{3}$ act on $T\cong T_{\omega_6}$ as $s\cdot z=\omega_3z$. Recall that $s\cdot \wp(z)=\wp(\omega_{3}^{-1}z)=\omega_3^2\wp(z)$ and $s\cdot \wp'=\wp'$ by Equations \eqref{wpScaling1} and \eqref{wpScaling2}.
By Proposition \ref{IsotypicalComponents} we have
$$
\mathcal{O}_{\mathbb{T}}^{\tilde{\sigma}(C_3)}=\mathcal{O}_{\mathbb{T}}^{\chi_0}=\CC[\wp'],\quad
\mathcal{O}_{\mathbb{T}}^{\chi_1}=\CC[\wp']\wp^2,\quad
\mathcal{O}_{\mathbb{T}}^{\chi_2}=\CC[\wp']\wp.
$$
Define the matrices $$H=\begin{pmatrix}
1& 0\\
0 &-1
\end{pmatrix},\quad E=\begin{pmatrix}
0 & 1\\
0 &0
\end{pmatrix}\otimes \wp,\quad F=\begin{pmatrix}
0 & 0\\
1 &0
\end{pmatrix}\otimes \wp^2,$$ which are easily seen to be invariant.
By decomposing the action of $C_3$ on both factors of the tensor product, we get $(\mf{sl}_2\otimes_{\CC}\ot)^{\rho\otimes \tilde{\sigma}(C_3)}=\bigoplus_{j=0}^2\mf{sl}_2^{\chi_j}\otimes_{\CC}\ot^{\overline{\chi_j}}$, where $\overline{\chi_j}(g)=\overline{\chi_j(g)}$, and we obtain $$\mathfrak{A}:=(\mathfrak{sl}_2\otimes_{\CC} \mathcal{O}_{\mathbb{T}})^{\rho\otimes\tilde{\sigma}(C_3)}= \CC\langle E,F,H\rangle\otimes_{\CC} \mathcal{O}_{\mathbb{T}}^{\tilde{\sigma}(C_3)}=\CC\langle E,F,H\rangle\otimes_{\CC} \CC[\wp']$$ with brackets $$[H,E]=2E, \quad [H,F]=-2F, \quad [E,F]=H\otimes \wp^3.$$
We will show that after performing a sequence of scalings and translations, $[E,F]$ can be written as $[E,F]=H\otimes \wp'(\wp'-1)$, after which we will recognise Onsager's algebra $\mf{O}$ in it.

By Lemma \ref{IsotypicalComponents}, a generator $J$ of $\mathcal{O}_{\mathbb{T}}^{\tilde{\sigma}(C_3)}$ can be chosen to be $\wp'$. The last bracket can be rewritten as 
\begin{align*}
[E,F]=H\otimes \frac{1}{4}((\wp')^2+g_3)=H\otimes \frac{1}{4}(J^2+g_3).
\end{align*}
Notice that $g_3\neq 0$ since $g_2=0$ and they cannot be both zero. Let $E'=-\frac{4}{g_3}E$. Then $[E',F]=H\otimes (-\frac{1}{g_3}J^2-1).$ Replace $-\frac{1}{g_3}J^2$ by $J^2$ to obtain $ [E',F]=H\otimes (J^2-1)$.
Replace $J-1$ by $J$ to get $[E',F]=H\otimes J(J+2)$ and transform further to obtain $[E',F]=H\otimes J(J-1)$ to see that $\mathfrak{A}$ is isomorphic to the Onsager algebra, cf. \cite[Theorem 2.5]{knibbeler2023automorphic}.\\

Another possibility for embedding $C_3$, is given by $\sigma'(r)(z)=\omega_3^{-1}z.$ The $\chi_1$ and $\chi_2$ isotypical components interchange when we consider this action. Define the elements $$\tilde{H}=-\begin{pmatrix}
1& 0\\
0 &-1
\end{pmatrix},\quad\tilde{E}=-\begin{pmatrix}
0 & 1\\
0 &0
\end{pmatrix}\otimes \wp^2,\quad \tilde{F}=-\begin{pmatrix}
0 & 0\\
1 &0
\end{pmatrix}\otimes \wp.$$
Observe how $\mathrm{Ad}\begin{pmatrix}
0 & 1\\
-1 &0
\end{pmatrix}$ maps $H,E,F$ to $\tilde{H},\tilde{F},\tilde{E}$, respectively. By Lemma \ref{lem:IsomALias}, we see that
\begin{align*}
(\mathfrak{sl}_2\otimes_{\CC} \mathcal{O}_{\mathbb{T}})^{\rho\otimes \tilde{\sigma}'(C_3)}\cong (\mathfrak{sl}_2\otimes_{\CC} \mathcal{O}_{\mathbb{T}})^{\rho\otimes \tilde{\sigma}(C_3)}.
\end{align*}
The cases of $\ell=4$ and $\ell=6$ follow in an analogous manner. By Example \ref{C2aLia}, $(\mathfrak{sl}_2\otimes_{\CC} \mathcal{O}_{\mathbb{T}})^{C_2}\not\cong \mf{O}$. This proves the claim.
\end{proof}
The proof of Theorem \ref{ClaLias} shows that a normal form of $(\mathfrak{sl}_2\otimes_{\CC}\mathcal{O}_{\mathbb{T}})^{\rho\otimes \tilde{\sigma}(C_{3})}$ is given by $\CC\langle E_3,F_3,H\rangle \otimes_{\CC} \CC[\wp']$, where $H=h$, $E_3=e\otimes \wp$, $F_3=f\otimes \wp^2,$ and $\wp=\wp_{\LL_{\omega_6}}$. 
In a similar fashion, one can obtain normal forms for remaining aLias with $\Gamma=C_4$ and $\Gamma=C_6$ using Proposition \ref{IsotypicalComponents}. For $\ell\in \{4,6\}$, choose the homomorphism $\rho_{\ell}:C_{\ell}\rightarrow \A(T_{\omega_{\ell}})$ defined by
$$\rho_{\ell}(s)\begin{pmatrix}
a & b\\
c &-a
\end{pmatrix}=\begin{pmatrix}
a & \omega_{\ell}b\\
\omega_{\ell}^{-1}c &-a
\end{pmatrix}$$
and $\sigma_{\ell}:C_{\ell}\rightarrow \A(T_{\omega_{\ell}})$ defined by $\sigma(s)z=\omega_{\ell}z$. Then  
\begin{align}\label{Normal Form C4}
(\mathfrak{sl}_2\otimes_{\CC}\mathcal{O}_{\mathbb{T}})^{\rho_4\otimes \tilde{\sigma}_4(C_{4})}=\CC\langle E_4,F_4,H\rangle \otimes_{\CC} \CC[\wp^2]\cong \mf{O},
\end{align} 
where $E_4=e\otimes \wp\wp'$, $F_4=f\otimes \wp'$, and $\wp=\wp_{\LL_i}$.\\

Finally, for $C_6$ we have the following normal form:
\begin{align}\label{Normal Form C6}
(\mathfrak{sl}_2\otimes_{\CC}\mathcal{O}_{\mathbb{T}})^{\rho_6\otimes \tilde{\sigma}_6(C_{6})}=\CC\langle E_6,F_6,H\rangle \otimes_{\CC} \CC[\wp^3]\cong \mf{O},
\end{align}
 where $E_6=e\otimes \wp\wp'$, $F_6=f\otimes \wp^2\wp'$, and $\wp=\wp_{\LL_{\omega_6}}$.
 
\subsection{The cases $\Gamma=C_2\times C_2$ and $\Gamma=A_4$}\label{subsectionC2xC2}

We will now consider the final two aLias in our classification, namely those with symmetry group $C_2\times C_2$ and a homomorphism $\sigma:C_2\times C_2\rightarrow \A(T)$, such that $g(T/\sigma(C_2\times C_2))=1$, where $T=\CC/\ZZ\oplus\ZZ\tau$, and those with symmetry group $A_4$.
We may assume that the homomorphism $\sigma$ is given by $\sigma(r_1)z=z+\frac{1}{2}$ and $\sigma(r_2)z=z+\frac{\tau}{2}$, where $r_1,r_2$ generate $C_2\times C_2$, cf. Lemma \ref{lem:groups up to conjugation}.\\

Let $\{\alpha_{00},\alpha_{01},\alpha_{10},\alpha_{11}\}$ denote the set of characters of $C_2\times C_2$, defined by $\alpha_{ij}(r_1,r_2)=\chi_i(r_1)\chi_j(r_2)$, where $\chi_0$, $\chi_1$ are the characters of $C_2$. Let $\mathbb{T}=T\setminus \mathcal{S}$, with $\mathcal{S}=\sigma(C_2\times C_2)\cdot \{0\}=\{0,1/2,\tau/2,(1+\tau)/2\}$ and write $\wp=\wp_{\LL}$. The divisor of $1/\wp'$ is given by 
\begin{align}\label{def:divisor 1/wp'}
(1/\wp')=-(1/2)-(\tau/2)-((1+\tau)/2))+3(0),
\end{align} and therefore $1/\wp'\in \mathcal{O}_{\mathbb{T}}$. Let $\pi_{\alpha_{ij}}=\frac{1}{4}\sum_{r\in C_2\times C_2}\overline{\alpha_{ij}(r)}r$ be the projection of $\mathcal{O}_{\mathbb{T}}$ onto the isotypical component $\mathcal{O}_{\mathbb{T}}^{\alpha_{ij}}$. Define $p_{ij}:=4\pi_{\alpha_{ij}}(1/\wp')$. Concretely, we have
\begin{align}
p_2:=p_{01}&=\frac{1}{\wp'} +r_1\frac{1}{\wp'}-r_2\frac{1}{\wp'}-r_1r_2\frac{1}{\wp'}, \label{p2}\\
p_1:=p_{10}&=\frac{1}{\wp'} -r_1\frac{1}{\wp'}+r_2\frac{1}{\wp'}-r_1r_2\frac{1}{\wp'}, \label{p1}\\
p_0:=p_{11}&=\frac{1}{\wp'} -r_1\frac{1}{\wp'}-r_2\frac{1}{\wp'}+r_1r_2\frac{1}{\wp'}. \label{p0}
\end{align}
By construction, the functions $p_{ij}$ are elements of $\mathcal{O}_{\mathbb{T}}^{\alpha_{ij}}$. Observe that they have at most order 1 poles in the set $\mathcal{S}$. Furthermore, they are odd functions since $\wp'$ is odd and $s$ (with $s(z)=-z$) and $r_1,r_2$ commute.
The definition of the $p_i$ is motivated in Lemma \ref{GammaIdentities}.

\begin{Remark}\label{JacobiTheta}
The $p_j$ can also be defined in terms of theta functions, as is done in \cite{carey1993landau}.
They define $$\lambda_1(z)=\frac{\vartheta_{00}(2z|\tau)}{\vartheta_{00}(0|\tau)\vartheta_{11}(2z|\tau)},\quad \lambda_2(z)=\frac{\vartheta_{10}(2z|\tau)}{\vartheta_{10}(0|\tau)\vartheta_{11}(2z|\tau)},\quad \lambda_3(z)=\frac{\vartheta_{01}(2z|\tau)}{\vartheta_{01}(0|\tau)\vartheta_{11}(2z|\tau)},$$
where $\vartheta_{ab}(z|\tau)=\sum_{n\in \ZZ}\exp(2\pi i[\tau(n+a/2)^2/2+(n+a/2)(z+b/2)])$, and where $(a,b)=(0,0), (1,0), (0,1), (1,1)$.
The relations between our projected $1/\wp'$ are, up to a scaling, as follows: $$p_0\leftrightarrow \lambda_1,\quad p_1\leftrightarrow \lambda_3, \quad p_2\leftrightarrow \lambda_2.$$
\end{Remark}

Written out explicitly, the group $C_2\times C_2$ acts on the $p_i$ as
\begin{align}
r_1\cdot p_0&=-p_0,& r_2\cdot p_0&=-p_0,\\
r_1\cdot  p_1&=- p_1, & r_2\cdot  p_1&= p_1,\\
r_1\cdot  p_2&= p_2, & r_2\cdot  p_2&=- p_2.
\end{align}
Let us obtain the series expansion of $\frac{1}{r\wp'}$ about $z=0$, for $r\in \langle r_1,r_2\rangle\setminus 1$, which we will use to derive the expansions of the $p_i^2$ about $z=0$ in Lemma \ref{lem:ExpansionsGammas}. First of all, for $r\in \langle r_1,r_2\rangle\setminus 1$, we find for $\frac{1}{r\wp'}$ the following expansion about $z=0$: 
\begin{align}\label{Expansion}
\frac{1}{r\wp'(z)}= \frac{1}{6\wp(r\cdot 0)^2-\frac{1}{2}g_2}\frac{1}{z}-\frac{2\wp(r\cdot 0)}{6\wp(r\cdot 0)^2-\frac{1}{2}g_2}z+O(z^2).
\end{align}
For example, for $r=r_1$, we can simplify this to 
$$\frac{1}{r_1\wp'(z)}= \frac{1}{2(e_1-e_3)(e_1-e_2)}\frac{1}{z}-\frac{e_1}{(e_1-e_3)(e_1-e_2)}z+O(z^2),$$
using that $6\wp(1/2)^2-g_2/2=2(e_1-e_3)(e_1-e_2)$. 
Notice that the denominators of the coefficients are indeed nonzero, since for lattices $\LL$, the \emph{discriminant} $\Delta(\LL):=g_2^3-27g_3^2$ is nonzero and can be shown to be equal to $16(e_1-e_2)^2(e_1-e_3)^2(e_2-e_3)^2$.\\

Let us argue that the $p_j$ are not identically zero. To this end, consider $p_{00}=4\pi_{\alpha_{00}}(1/\wp')=\frac{1}{\wp'} +r_1\frac{1}{\wp'}+r_2\frac{1}{\wp'}+r_1r_2\frac{1}{\wp'}$.
 This function is invariant under $C_2\times C_2$ and hence descends to a function on $\tilde{T}:=T/\sigma(C_2\times C_2)$. Now, $p_{00}$ has at most order 1 poles and on $\tilde{T}$, there is therefore at most a single pole. Thus $p_{00}$ must be constant. In fact, since $\wp'$ is odd, $p_{00}$ is odd as well. Hence $p_{00}=0.$
Knowing that $p_{00}$ is constant allows us to draw conclusions about the $p_i$. 
By considering the divisor of $\frac{1}{\wp'}$ in \eqref{def:divisor 1/wp'}, we see that the sets of poles of $r_1\frac{1}{\wp'}, r_2\frac{1}{\wp'}$ and $r_1r_2\frac{1}{\wp'}$ have pairwise precisely two poles in common. We therefore see that the three functions 
\begin{align*}
k_1:=r_2\frac{1}{\wp'}+r_1r_2\frac{1}{\wp'}, \quad k_2:=r_1\frac{1}{\wp'}+r_1r_2\frac{1}{\wp'}, \quad k_3:=r_1\frac{1}{\wp'}+r_2\frac{1}{\wp'}
\end{align*}
are all non-constant. As a consequence we get that $p_0=p_{00}-2k_3=-2k_3$, $ p_1=p_{00}-2k_2=-2k_2$ and $ p_2=p_{00}-2k_1=-2k_1$ are all nonzero. \\

The next lemma describes the algebra of $C_2\times C_2$-invariant subspace of $\mathcal{O}_{\mathbb{T}}$.
\begin{Lemma}\label{InvariantsC2xC2}
We have $\mathcal{O}_{\mathbb{T}}^{\tilde{\sigma}(C_2\times C_2)}=\CC[\wp_{\frac{1}{2}\LL},\wp_{\frac{1}{2}\LL}']$.
\end{Lemma}
\begin{proof}
By definition, $\mathcal{O}_{\mathbb{T}}^{\tilde{\sigma}(C_2\times C_2)}$ is the algebra of all $\frac{1}{2}\LL$-periodic meromorphic functions on $T$ that are holomorphic on $T\setminus\frac{1}{2}\LL$, which corresponds to $\CC[\wp_{\frac{1}{2}\LL},\wp_{\frac{1}{2}\LL}']$ by Lemma \ref{RegularFunctions}.
\end{proof}
We now record the relations between the $p_i^2$ and $\wp_{\frac{1}{2}\LL}$ in a lemma for later purposes.
\begin{Lemma}\label{lem:ExpansionsGammas}
We have
\begin{align*}
p_0^2&=\frac{1}{(e_1-e_3)^2(e_2-e_3)^2}(\wp_{\frac{1}{2}\LL}-4e_3)=\frac{16(e_2-e_3)^2}{\Delta(\LL)}(\wp_{\frac{1}{2}\LL}-4e_3).
\end{align*}
Similarly, the relations of $p_j^2$ and $\wp_{\frac{1}{2}\LL}$ for $j=1,2$ are obtained from $p_0^2$ by cyclic permutations of the $e_i$. 
Consequently, there are the following relations: $ p_1^2=\alpha_1 p_2^2+\alpha_2$, $ p_0^2=\beta_1 p_2^2+\beta_2$, where $$\alpha_1=\left(\frac{e_1-e_3}{e_2-e_3}\right)^2,\quad \alpha_2=\frac{4}{(e_1-e_2)(e_2-e_3)^2},$$ and $$\beta_1=\left(\frac{e_1-e_2}{e_2-e_3}\right)^2,\quad \beta_2=\frac{4}{(e_1-e_3)(e_2-e_3)^2}.$$ Thus $\alpha_1,\alpha_2,\beta_1,\beta_2\in \CC^*$.
\end{Lemma}

\begin{proof}
Using the expansion in \eqref{Expansion}, we obtain for $r,\tilde{r}\in \langle r_1,r_2\rangle\setminus 1$: 
\begin{align*}\left(\frac{1}{r\wp'(z)}\right)\left(\frac{1}{\tilde{r}\wp'(z)}\right)&=\frac{1}{(6\wp(r\cdot 0)^2-\frac{1}{2}g_2)(6\wp(\tilde{r}\cdot 0)^2-\frac{1}{2}g_2)}\frac{1}{z^2}-\\
&\quad \frac{2\wp(r\cdot 0)+2\wp(\tilde{r}\cdot 0)}{(6\wp(r\cdot 0)^2-\frac{1}{2}g_2)(6\wp(\tilde{r}\cdot 0)^2-\frac{1}{2}g_2)}+O(z^2),
\end{align*}
where we use that $\wp'$ is odd, and hence the above product is even.
Hence, for $r=r_1$ and $\tilde{r}=r_2$, we get \begin{align*}
p_0(z)^2=4k_3(z)^2&=4\left(\frac{1}{r_1\wp'(z)}\right)^2+4\left(\frac{1}{r_2\wp'(z)}\right)^2+8\left(\frac{1}{r_1\wp'(z)}\right)\left(\frac{1}{r_2\wp'(z)}\right)\\
&=\frac{1}{(e_1-e_3)^2(e_2-e_3)^2}\frac{1}{z^2}-\frac{4e_3}{(e_1-e_3)^2(e_2-e_3)^2}+O(z^2).
\end{align*}
Similarly, one can show that $p_j^2$ $(j=1,2)$ is obtained from $p_0^2$ by changing $e_i$ to $e_{i-j}$ in the expansions for $i=1,2,3$.

It is clear from the definition of $\wp$ as given in \eqref{DefWp}, that about $z=0$, $\wp$ has the expansion $\wp(z)=\frac{1}{z^2}+O(z^2)$. Since $p_j^2$ has order 2 poles and is invariant under $C_2\times C_2$, we know that $p_j^2=c_1\wp_{\frac{1}{2}\LL}+c_2$ for some $c_1,c_2\in \CC$. The last statement follows now directly. This completes the proof.
\end{proof}
The symmetry groups for $\mathfrak{sl}_2$-based aLias having $C_2\times C_2$ as a subgroup of translations, are limited to $A_4\cong C_3\ltimes(C_2\times C_2)$ (for hexagonal lattices) and $C_2\times C_2$ itself. We are therefore interested in how $C_3$ acts on functions in $\mathcal{O}_{\mathbb{T}}^{C_2\times C_2}$. This is the content of the following lemma. 
\begin{Lemma}\label{GammaIdentities}
Suppose $\LL$ is homothetic to $\LL_{\omega_6}$ and $s$ acts on $\CC/\LL$ as $s\cdot z=\omega_3z$.  Then $$s\cdot p_j=p_{j+1},$$
where the subscripts are taken modulo $3$.
Furthermore, $\alpha_1=\omega_3, \beta_1=\omega_3^2$ if and only if $\LL$ is homothetic to $\LL_{\omega_6}$.
\end{Lemma}
\begin{proof}
Assume that $\LL$ is homothetic to $\LL_{\omega_6}.$ 
Recall $\langle s\rangle\ltimes \langle r_1,r_2\rangle \cong C_3\ltimes (C_2\times C_2)\cong A_4$, where $s(z)=\omega_3z$ and $r_1(z)=z+1/2$ and $r_2(z)=z+\omega_3/2$, cf. Lemma \ref{lem:groups up to conjugation}. The following relations hold: $s r_1=r_1r_2s$ and $s r_2=r_1s.$
Using that $s\cdot \wp'=\wp'$, we compute 
$$
s\cdot  p_2=s\frac{1}{\wp'} +sr_1\frac{1}{\wp'}-sr_2\frac{1}{\wp'}-sr_1r_2\frac{1}{\wp'}=\frac{1}{\wp'} +r_1r_2\frac{1}{\wp'}-r_1\frac{1}{\wp'}-r_2\frac{1}{\wp'}= p_0.
$$
The other identities follow in the same way.\\

We now prove the second claim. We have $$ p_2^2=s\cdot  p_1^2=\alpha_1s\cdot  p_2^2+\alpha_2=\alpha_1 p_0^2+\alpha_2=\alpha_1\beta_1 p_2^2+\alpha_1\beta_2+\alpha_2$$ and we thus see that $\alpha_1\beta_1=1$ and $\alpha_1\beta_2=-\alpha_2$. By Lemma \ref{lem:ExpansionsGammas} we see $ p_0^2+ p_1^2+ p_2^2=(1+\alpha_1+\beta_1) p_2^2+\alpha_2+\beta_2$. Therefore the function $ p_0^2+ p_1^2+ p_2^2$ is constant if and only if $g_2=0$. Indeed, 
a computation shows that $1+\alpha_1+\beta_1=\frac{g_2}{2(e_2-e_3)^2}.$  We conclude that $\alpha_1$ and $\beta_1$ are distinct third roots of unity if and only if $\LL$ is homothetic to a hexagonal lattice. Using that $e_1^2+e_1e_2+e_2^2=0$ if and only if $\LL$ is homothetic to $\LL_{\omega_6}$, one can show that $\alpha_1=(2e_1+e_2)^2/(2e_2+e_1)^2=e_2^2/e_1^2.$
Now, $e_1=\omega_3 e_2$ by \eqref{wpScaling2}. Thus, we get $\alpha_1=\omega_3$ and $\beta_1=\omega_3^2$.
\end{proof}

Define \begin{align}\label{ValuesAB}
A_1=\left(\frac{e_1-e_3}{e_2-e_3}\right)^{\frac{3}{2}},\quad B_1=\left(\frac{e_1-e_2}{e_2-e_3}\right)^{\frac{3}{2}}.
\end{align} 
We note that $$\frac{e_1-e_2}{e_2-e_3}=\frac{\vartheta_{01}^4(0|\tau)}{\vartheta_{00}^4(0|\tau)}=1-\lambda(\tau),$$
where $\vartheta_{00},\vartheta_{01}$ are the Jacobi theta functions defined in Remark \ref{JacobiTheta} and $\lambda$ is the \emph{modular lambda function}, see for example \cite[Section 7.2]{rankin1977modular}.\\

Also, $\frac{e_1-e_3}{e_2-e_3}=1+\frac{e_1-e_2}{e_2-e_3},$ and hence $\frac{e_1-e_3}{e_2-e_3}=2-\lambda(\tau).$ In particular, we can express $\alpha_1$, $\beta_1$, $A_1$ and $B_1$ in terms of $\lambda$ as follows: $\alpha_1=(2-\lambda)^2$, $\beta_1=(1-\lambda)^2$, $A_1=(2-\lambda)^{3/2}$ and $B_1=(1-\lambda)^{3/2}$.\\ 

For later reference, we will record the following identities. 
\begin{Lemma}\label{GammaIdentities2}
The following identities hold: 
\begin{align*}
&\left(\frac{A_1}{\alpha_1} p_1+\frac{B_1}{\beta_1} p_0\right)\left(\frac{A_1}{\alpha_1} p_1-\frac{B_1}{\beta_1} p_0\right)= p_2^2,\\ 
&\left(\frac{A_1}{\alpha_1} p_1\mp\frac{B_1}{\beta_1} p_0\right)(A_1 p_0\pm B_1 p_1)= p_0 p_1\pm\sqrt{\alpha_2\beta_2}.
\end{align*}
Furthermore, if $\LL$ is homothetic to $\LL_{\omega_6}$, then $A_1^2=-1$ and $B_1^2=1$.
\end{Lemma}
\begin{proof}
Using Lemma \ref{GammaIdentities}, we compute 
$$
\left(\frac{A_1}{\alpha_1} p_1+\frac{B_1}{\beta_1} p_0\right)\left(\frac{A_1}{\alpha_1} p_1-\frac{B_1}{\beta_1} p_0\right)=\left(\frac{A_1^2}{\alpha_1}-\frac{B_1^2}{\beta_1}\right) p_2^2+\left(\frac{A_1^2\alpha_2}{\alpha_1^2}-\frac{B_1^2\beta_2}{\beta_1^2}\right).
$$
Another computation shows that $$\frac{A_1^2\alpha_2}{\alpha_1^2}=\frac{4}{(e_1-e_3)(e_2-e_3)(e_1-e_2)}=\frac{B_1^2\beta_2}{\beta_1^2}$$ 
and $$\frac{A_1^2}{\alpha_1}-\frac{B_1^2}{\beta_1}=\frac{e_1-e_3}{e_2-e_3}-\frac{e_1-e_2}{e_2-e_3}=1.$$
This proves the first claim. The other identities follow by similar calculations, where we note that the sign of $\sqrt{\alpha_2\beta_2}$ comes from the choice of the square roots in the definition of $A_1$ and $B_1$. Notice that $\sqrt{\alpha_2\beta_2}\neq 0$ since $\Delta(\LL)=16(e_1-e_3)^2(e_2-e_3)^2(e_1-e_2)^2\neq 0$ for any lattice $\LL$. The final claim follows from the definition of $A_1$ and $B_1$, cf. \eqref{ValuesAB}, and using that in this case, $e_1=\omega_3e_2$.
\end{proof}
We will now define two functions which arise in the construction of an intertwiner $\Psi$. The intertwiner will be used to construct a basis for the $C_2\times C_2$-aLia with $g(T/C_2\times C_2)=1$. 
Define \begin{align}\label{xiDef}
\xi_{\pm}:=\sqrt{\frac{A_1}{\alpha_1} p_1\pm\frac{B_1}{\beta_1} p_0}=\sqrt{\sqrt{\frac{e_2-e_3}{e_1-e_3}} p_1\pm\sqrt{\frac{e_2-e_3}{e_1-e_2}} p_0},
\end{align} 
where we choose the branches in such way that $\xi_-\xi_+=p_2$ (having Lemma \ref{GammaIdentities2} in mind).
Notice that $\xi_{\pm}$ are not functions on any complex torus. \\
 
We will now turn our attention to finding an intertwiner $\Ad(\Phi)$ as we have done for $\Gamma=C_N$, cf. Section \ref{subsectionCNDN}. Recall that we constructed an $SL_2(\CC)$-valued map $\Phi$ such that $\Phi(r\cdot z)=\delta(r)\Phi(z)$, where $r\in C_N$ and $\delta$ is a monomorphism $C_N\rightarrow SL_2(\CC)$. Composing $\Phi$ with $\Ad$ yields a map equivariant with respect to the action of $C_N$ on a torus $T$ and $\mf{sl}_2$. Let us now argue that a similar construction, that is, obtaining an intertwiner via conjugation of some $C_N$-equivariant $SL_2(\CC)$-valued map, cannot be carried out for $\Gamma=C_2\times C_2$. A necessary condition for carrying out this construction, is that $\delta:\Gamma\rightarrow SL_2(\CC)$ should be a monomorphism. However,
by considering the character table of $C_2\times C_2$, one concludes that there does not exist a faithful representation $C_2\times C_2\rightarrow SL_2(\CC)$, and hence there is no monomorphism $\delta$ such that the above condition holds for nonzero $\Phi$. However,  there does exist a faithful representation $C_2\times C_2\rightarrow \A(\mf{sl}_2)$, which is the context in which we will look for an intertwiner below.\\

Define $\rho: C_2\times C_2\rightarrow \A(\mf{sl}_2)$ by 
\begin{align} \label{def:C2xC2action}
\rho(r_1)\begin{pmatrix}
a & b\\
c &-a
\end{pmatrix}=\begin{pmatrix}
a & -b\\
-c &-a
\end{pmatrix}, &\quad
\rho(r_2)\begin{pmatrix}
a & b\\
c &-a
\end{pmatrix}=\begin{pmatrix}
-a & -c\\
-b &a
\end{pmatrix}.
\end{align}
\begin{Remark}\label{rem:Q8}
We have used the quaternions $Q_8$ as a double cover of $C_2\times C_2$ in \eqref{def:C2xC2action}. The matrices $R_1=\begin{pmatrix}
i& 0\\
0&-i
\end{pmatrix}$ and $R_2=\begin{pmatrix}
0& 1\\
-1&0
\end{pmatrix}$ generate $Q_8$ inside $SL_2(\CC)$, which define a homomorphism $\zeta:Q_8\rightarrow SL_2(\CC)$. Then $\Ad(\zeta(Q_8))\cong C_2\times C_2$ inside $\A(\mf{sl}_2)$.
\end{Remark}

Introduce the following matrix, which will play an important role in the construction of the intertwiner:
\begin{align}\label{MatrixP}
\Omega=\begin{pmatrix}
\sqrt{\frac{A_1}{\alpha_1} p_1-\frac{B_1}{\beta_1} p_0}&\frac{1}{2\sqrt{\alpha_2\beta_2}}(A_1 p_0-B_1 p_1)\sqrt{\frac{A_1}{\alpha_1} p_1+\frac{B_1}{\beta_1} p_0}\\
\sqrt{\frac{A_1}{\alpha_1} p_1+\frac{B_1}{\beta_1} p_0}&\frac{1}{2\sqrt{\alpha_2\beta_2}}(A_1 p_0+B_1 p_1)\sqrt{\frac{A_1}{\alpha_1} p_1-\frac{B_1}{\beta_1} p_0}
\end{pmatrix},
\end{align}
where we recall that $\alpha_j,\beta_j$ are defined in Lemma \ref{GammaIdentities} and $A_1,B_1$ in \eqref{ValuesAB}.\\

The following two simple lemmata establish that $\Ad(\Omega)$ is a $C_2\times C_2$-equivariant automorphism of $\mathfrak{sl}_2\otimes_{\CC}\mathcal{O}_{\mathbb{T}}$.
First of all, conjugation with $\Omega$ is an automorphism of $\mathfrak{sl}_2\otimes_{\CC}\mathcal{O}_{\mathbb{T}}$: 

\begin{Lemma}\label{lem:def:M}
 We have $\Psi:=\Ad(\Omega)\in \A_{\mathcal{O}_{\mathbb{T}}}(\mathfrak{sl}_2\otimes_{\CC}\mathcal{O}_{\mathbb{T}}).$ Explicitly, the matrix of $\Psi$ with respect to the basis $\mathcal{B}=\{h,e,f\}$, is given by
\begin{align*}[\Psi]_{\mathcal{B}}=\begin{pmatrix}
\frac{1}{\sqrt{\alpha_2\beta_2}} p_0 p_1 & - p_2 & \frac{1}{4\alpha_2\beta_2}(A_1^2 p_0^2-B_1^2 p_1^2) p_2\\
\frac{1}{\sqrt{\alpha_2\beta_2}}(B_1 p_1-A_1 p_0) p_2 & \frac{A_1}{\alpha_1} p_1-\frac{B_1}{\beta_1} p_0 & \frac{1}{4\alpha_2\beta_2}(B_1 p_1-A_1 p_0)\tilde{p}_-\\
\frac{1}{\sqrt{\alpha_2\beta_2}}(A_1 p_0+B_1 p_1) p_2 & -\frac{A_1}{\alpha_1} p_1-\frac{B_1}{\beta_1} p_0 &\frac{1}{4\alpha_2\beta_2} (A_1 p_0+B_1 p_1)\tilde{p}_+
\end{pmatrix},
\end{align*}
where $\tilde{p}_{\pm}:=p_0 p_1\pm \sqrt{\alpha_2\beta_2}$.
\end{Lemma}
\begin{proof}
We will first argue that $\det(\Omega)=1$. Recall $\xi_{\pm}=\sqrt{\frac{A_1}{\alpha_1} p_1\pm\frac{B_1}{\beta_1} p_0}$. Using Lemma \ref{GammaIdentities2}, we compute
\begin{align*}
\det(\Omega)&=\frac{1}{2\sqrt{\alpha_2\beta_2}}(\xi_-^2(A_1 p_0+B_1 p_1)-\xi_+^2(A_1 p_0-B_1p_1))\\
&=\frac{1}{2\sqrt{\alpha_2\beta_2}}(( p_0 p_1+\sqrt{\alpha_2\beta_2})-( p_0 p_1-\sqrt{\alpha_2\beta_2}))\\
&=1.
\end{align*}
Let $m=\begin{pmatrix}a&b\\c&d\end{pmatrix}\in SL_2(R)$, where $R$ is some unital ring. The matrix of $\mathrm{Ad}(m)$ with respect to $\mathcal{B}$ is given by $$[\mathrm{Ad}(m)]_{\mathcal{B}}=\begin{pmatrix}bc+ad&-ac&bd\\
-2ab&a^2&-b^2\\
2cd&-c^2&d^2\end{pmatrix}\in SL_3(R).$$
Take $m=\Omega$ and after some work, using Lemma \ref{GammaIdentities2}, one arrives at the matrix in the claim. Since $\det(\Omega)=1$, the same holds for $\Psi $ and we see that all the entries are in $\mathcal{O}_{\mathbb{T}}$. This shows that $\Psi\in \A_{\mathcal{O}_{\mathbb{T}}}(\mathfrak{sl}_2\otimes_{\CC}\mathcal{O}_{\mathbb{T}}).$
\end{proof}

\begin{Lemma}\label{EquivarianceProperty}
The map $\Psi$ intertwines the actions of $C_2\times C_2$ on $T$ and $\mf{sl}_2$: $\Psi (r\cdot z)=\rho(r)\Psi (z)$ for all $r\in C_2\times C_2$ and $z\in \TT$.
\end{Lemma}
\begin{proof}
Using that $$\left[\mathrm{Ad}\begin{pmatrix}
i& 0\\
0&-i
\end{pmatrix}\right]_{\mathcal{B}}=\begin{pmatrix}
1&0&0\\
0&-1&0\\
0&0&-1
\end{pmatrix},\quad \left[\mathrm{Ad}\begin{pmatrix}
0& 1\\
-1&0
\end{pmatrix}\right]_{\mathcal{B}}=\begin{pmatrix}
-1&0&0\\
0&0&-1\\
0&-1&0
\end{pmatrix},$$ 
it is straightforward to verify the claim for $r_1$ and $r_2$, and hence for all $r\in C_2\times C_2$.
\end{proof}
\begin{Remark}
Notice that in the definition of $\Omega$ in \eqref{MatrixP}, the choice of signs of the constants $A_1$ and $B_1$ does not matter. Indeed, one could verify that changing the signs corresponds to compose $\Psi=\Ad(\Omega)$ with commuting (order 2) elements of $\A(\mf{sl}_2)$, which in turn also commute with $\rho(r)$, for any $r\in C_2\times C_2$.
\end{Remark}
Define the following elements of $\mathfrak{sl}_2\otimes_{\CC}\mathcal{O}_{\mathbb{T}}$: 
\begin{align}\label{C2xC2Generators}
h'=\Psi h,\quad e'=\Psi e, \quad f'=\Psi f. 
\end{align}
By the equivariance property of $\Psi $, cf. Lemma \ref{EquivarianceProperty}, the elements $h',e'$ and $f'$ are invariant with respect to the action of $C_2\times C_2$. In fact, they will form a basis of the $C_2\times C_2$-aLia as we will see in the following theorem.
\begin{Theorem}\label{C2xC2aLia}
Let $\rho:C_2\times C_2\rightarrow \A(\mf{sl}_2)$ and $\sigma:C_2\times C_2\rightarrow \A(T)$ be monomorphisms and assume that $g(T/\sigma(C_2\times C_2))=1$. Then
$$(\mathfrak{sl}_2\otimes_{\CC}\mathcal{O}_{\mathbb{T}})^{\rho\otimes\tilde{\sigma}(C_2\times C_2)}\cong \mf{sl}_2\otimes_{\CC}\ot^{\tilde{\sigma}(C_2\times C_2)}.$$
A normal form is given by $$(\mathfrak{sl}_2\otimes_{\CC}\mathcal{O}_{\mathbb{T}})^{\rho\otimes \tilde{\sigma}(C_2\times C_2)}=\CC\langle h',e',f' \rangle \otimes_{\CC}\CC[\wp_{\frac{1}{2}\LL},\wp_{\frac{1}{2}\LL}'] $$ with $$[h',e']=2e',\quad [h',f']=-2f',\quad [e',f']=h',$$ where $h',e',f'$ are defined in \eqref{C2xC2Generators}.
\end{Theorem}
\begin{proof}
By Lemma \ref{InvariantsC2xC2} , we know that $\mathcal{O}_{\mathbb{T}}^{\tilde{\sigma}(C_2\times C_2)}=\CC[\wp_{\frac{1}{2}\Lambda},\wp_{\frac{1}{2}\Lambda}']$. Now, by Lemma \ref{EquivarianceProperty}, we are in a position to apply Proposition \ref{Mauto} to $\Psi=\Ad(\Omega) $ with $K=C_2\times C_2$. Hence 
$$(\mathfrak{sl}_2\otimes_{\CC}\mathcal{O}_{\mathbb{T}})^{\rho\otimes\tilde{\sigma}(C_2\times C_2)}=\CC\langle h',e',f' \rangle \otimes_{\CC}\CC[\wp_{\frac{1}{2}\LL},\wp_{\frac{1}{2}\LL}'] .$$ The Lie structure is the same as the Lie structure of $\CC\langle e,f,h\rangle$ and is therefore given by  $$[h',e']=2e',\quad [h',f']=-2f',\quad [e',f']=h'.$$ In particular, we see that $(\mathfrak{sl}_2\otimes_{\CC}\mathcal{O}_{\mathbb{T}})^{\rho\otimes \tilde{\sigma}(C_2\times C_2)}\cong \mf{sl}_2\otimes_{\CC}\CC[\wp_{\frac{1}{2}\LL},\wp_{\frac{1}{2}\LL}']= \mf{C}_{\frac{1}{2}\LL}.$
\end{proof}
Explicitly, the generators of the $C_2\times C_2$-aLia in normal form, are given by \begin{align*}
h'&=\frac{1}{\sqrt{\alpha_2\beta_2}}\begin{pmatrix}
 p_0 p_1 & -\left(\frac{e_1-e_3}{e_2-e_3}\right)^{\frac{3}{2}} p_0 p_2+\left(\frac{e_1-e_2}{e_2-e_3}\right)^{\frac{3}{2}} p_1 p_2\\
\left(\frac{e_1-e_3}{e_2-e_3}\right)^{\frac{3}{2}} p_0 p_2+\left(\frac{e_1-e_2}{e_2-e_3}\right)^{\frac{3}{2}} p_1 p_2 & - p_0 p_1
\end{pmatrix},\\
 e'&=\begin{pmatrix}
- p_2&\sqrt{\frac{e_2-e_3}{e_1-e_3}} p_1-\sqrt{\frac{e_2-e_3}{e_1-e_2}} p_0\\
-\sqrt{\frac{e_2-e_3}{e_1-e_3}} p_1-\sqrt{\frac{e_2-e_3}{e_1-e_2}} p_0 &  p_2
\end{pmatrix},\\
f'&=\frac{1}{4\alpha_2\beta_2}\begin{pmatrix}
\left(\frac{\Delta(\LL)}{16(e_2-e_3)^4} p_2^2+\frac{12e_1}{(e_2-e_3)^4}\right) p_2& \left(\left(\frac{e_1-e_2}{e_2-e_3}\right)^{\frac{3}{2}} p_1-\left(\frac{e_1-e_3}{e_2-e_3}\right)^{\frac{3}{2}} p_0\right)\tilde{p}_-\\
 \left(\left(\frac{e_1-e_3}{e_2-e_3}\right)^{\frac{3}{2}} p_0+\left(\frac{e_1-e_2}{e_2-e_3}\right)^{\frac{3}{2}} p_1\right)\tilde{p}_+ & -\left(\frac{\Delta(\LL)}{16(e_2-e_3)^4} p_2^2+\frac{12e_1}{(e_2-e_3)^4}\right) p_2\end{pmatrix}.
\end{align*}
where we recall that $\tilde{p}_{\pm}=p_0 p_1\pm\sqrt{\alpha_2\beta_2}$.
We could rewrite/simplify the term $\frac{\Delta(\LL)}{16(e_2-e_3)^4} p_2^2$ to $\wp_{\frac{1}{2}\LL}-4e_1$ due to Corollary \ref{lem:ExpansionsGammas}.
\begin{Remark}\label{rem:LL}
The aLia $\mf{A}=(\mathfrak{sl}_2\otimes_{\CC}\mathcal{O}_{\mathbb{T}})^{\rho\otimes\tilde{\sigma}(C_2\times C_2)}$ is the Lie algebra of holomorphic $K=C_2\times C_2$-equivariant maps from a punctured torus to $\mf{sl}_2$. In \cite{carey1993landau}, the authors study holomorphic $K$-equivariant maps from some $K$-invariant subset of a complex torus $\CC/\ZZ\oplus\ZZ\tau$ to $SL_2(\CC)$. They give expansions of the functions that appear in these matrices in terms of $\lambda_1,\lambda_2$ and $\lambda_3$, which are related to our functions $p_0,p_1$ and $p_2$ defined in \eqref{p2}-\eqref{p0}. With these functions, they construct $K$-equivariant maps with values in $\mf{sl}_2$ which they use in the context of integrable systems, in particular in an algebraic description of the Landau-Lifshitz hierarchy. For example, a Lax pair, which is the Lax pair found in \cite{sklyanin1979complete}, is constructed for the Landau-Lifshitz equation, using elements of $\mf{A}$. Furthermore, they present a result about a factorisation of smooth $K$-equivariant loops $\gamma:\mathcal{C}\rightarrow SL_2(\CC)$, where $\mathcal{C}$ is a union of disjoint circles around the points $0,1/2,\tau/2,(1+\tau)/2$, analogues to Birkhoff factorisation.
\end{Remark}
 
The last symmetry group in our classification is $\Gamma=A_4$, which we will discuss now.
\begin{Theorem}\label{A4aLia}
Let $T\cong T_{\omega_6}$ and $\rho:A_4\rightarrow \A(\mf{sl}_2)$ and $\sigma:A_4\rightarrow \A(T)$ be monomorphisms. Let $\mathbb{T}=T\setminus A_4\cdot \{0\}.$ There is the following isomorphism of Lie algebras: $$(\mathfrak{sl}_2\otimes_{\CC} \mathcal{O}_{\mathbb{T}})^{\rho\otimes \tilde{\sigma}(A_4)}\cong\mathfrak{O},$$  where $\mathfrak{O}$ is the Onsager algebra. A normal form is given by $$(\mathfrak{sl}_2\otimes_{\CC} \mathcal{O}_{\mathbb{T}})^{\rho\otimes \tilde{\sigma}(A_4)}=\CC\langle e'\otimes \wp_{\frac{1}{2}\LL_{\omega_6}},f'\otimes \wp_{\frac{1}{2}\LL_{\omega_6}}^2,h'\rangle \otimes_{\CC} \CC[\wp_{\frac{1}{2}\LL_{\omega_6}}'],$$
where $e',f',h'$ are the generators of the $C_2\times C_2$-aLia, cf. Theorem \ref{C2xC2aLia}. 
\end{Theorem}
\begin{proof}
Assume $A_4$ is generated by $r_1,r_2$ and $s$ as in Lemma \ref{lem:groups up to conjugation} and notice that $A_4\cong C_3\ltimes (C_2\times C_2)$. 
Define $\rho:A_4\rightarrow \A(\mf{sl}_2)$ by $$\rho(s)=\Ad\left(\frac{1}{2}\begin{pmatrix}
 1+i &-1+i\\
 1+i&1-i
 \end{pmatrix}\right),$$
and $\rho(r_1)$, $\rho(r_2)$ as we have done for $C_2\times C_2$, see \eqref{def:C2xC2action}.  Define $\sigma:A_4\rightarrow \A(\mf{sl}_2)$ by $\sigma(s)z=\omega_3z$, $\sigma(r_1)z=z+1/2$ and $\sigma(r_2)z=z+\omega_3/2$. 
Theorem \ref{C2xC2aLia} gives us the normal form $\CC\langle h',e',f' \rangle \otimes_{\CC}\CC[\wp_{\frac{1}{2}\LL},\wp_{\frac{1}{2}\LL}']$ for the $C_2\times C_2$-aLia. 
We will verify that $h'$ is $C_3$-invariant, where $C_3\subset \A(\mf{sl}_2)$ is generated by $\rho(s)$, after which we can invoke Lemma \ref{Nnormal-form} and Proposition \ref{Mauto}. By Lemma \ref{IsotypicalComponents}, we know that $\CC[\wp_{\frac{1}{2}\LL_{\omega_6}},\wp_{\frac{1}{2}\LL_{\omega_6}}']^{\langle s\rangle}\cong \CC[\wp_{\frac{1}{2}\LL_{\omega_6}}'].$ Again, there is no loss of generality by our choice of $\rho$ and $\sigma$ by Lemma \ref{lem:Klein} combined with Lemma \ref{lem:IsomALias}.\\

Recall that by Lemma \ref{GammaIdentities2} $A_1^2=-1$ and $B_1^2=1$ in case of a hexagonal lattice. We choose $A_1=-i$ and $B_1=1$. The intertwiner $\Psi$ with respect to the basis $\{h,e,f\}$ becomes $$\begin{pmatrix}
\frac{1}{\sqrt{\alpha_2\beta_2}} p_0 p_1 & - p_2 & -\frac{1}{4\alpha_2\beta_2}(p_0^2+ p_1^2) p_2\\
\frac{1}{\sqrt{\alpha_2\beta_2}}(i p_0+ p_1) p_2 & -i\omega_3^2 p_1-\omega_3p_0 & \frac{1}{4\alpha_2\beta_2}(i p_0+ p_1)( p_0 p_1-\sqrt{\alpha_2\beta_2})\\
\frac{1}{\sqrt{\alpha_2\beta_2}}(-i p_0+ p_1) p_2 & i\omega_3^2 p_1-\omega_3 p_0 &\frac{1}{4\alpha_2\beta_2} (-i p_0+p_1)( p_0 p_1+\sqrt{\alpha_2\beta_2})
\end{pmatrix}.$$
We will show, given our choice of $A_1$ and $B_1$, that $\rho(s)\Psi (z)h=\Psi (s\cdot z)h$, which will imply that $\Psi h$ is $C_3$-invariant.
Write $[\cdot]$ for a vector or matrix with respect to the basis $\mathcal{B}=\{h,e,f\}$ and let $$U:=\left[\rho(s)\right]=\frac{1}{2}\begin{pmatrix}
0&-i&i\\
2&i&i\\
2&-i&-i
\end{pmatrix}.$$
We compute \begin{align*}U[\Psi (z)h]&=\frac{1}{2\sqrt{\alpha_2\beta_2}}\begin{pmatrix}
0&-i&i\\
2&i&i\\
2&-i&-i
\end{pmatrix} \begin{pmatrix}
p_0(z) p_1(z) \\
i p_0(z)p_2(z)+ p_1(z)p_2(z)\\
-i p_0(z)p_2(z)+ p_1(z)p_2(z)\end{pmatrix}\\
&=\frac{1}{\sqrt{\alpha_2\beta_2}}\begin{pmatrix}p_0(z)p_2(z)\\
p_0(z)p_1(z)+ip_1(z)p_2(z)\\
p_0(z)p_1(z)-ip_1(z)p_2(z) \end{pmatrix},
\end{align*}
and, using $p_j(s\cdot z)=p_{j-1}(z)$, 
\begin{align*}
[\Psi (s\cdot z)h]&=\frac{1}{\sqrt{\alpha_2\beta_2}}\begin{pmatrix}p_0(s\cdot z)p_1(s\cdot z)\\
i p_0(s\cdot z)p_2(s\cdot z)+ p_1(s\cdot z)p_2(s\cdot z)\\
-i p_0(s\cdot z)p_2(s\cdot z)+ p_1(s\cdot z)p_2(s\cdot z)\end{pmatrix}\\
&=\frac{1}{\sqrt{\alpha_2\beta_2}}\begin{pmatrix}p_0(z)p_2(z)\\
p_0(z)p_1(z)+ip_1(z)p_2(z)\\
p_0(z)p_1(z)-ip_1(z)p_2(z) \end{pmatrix}.
\end{align*}
This means that $h':=\Psi h$ is indeed invariant under the action of $C_3$. By Lemma \ref{Nnormal-form} and Proposition \ref{Mauto}, we see that $$(\mathfrak{sl}_2\otimes_{\CC} \mathcal{O}_{\mathbb{T}})^{A_4} =\CC\langle E,F,H\rangle\otimes_{\CC}\CC[\wp_{\frac{1}{2}\LL_{\omega_6}}'],$$
with generators $$E=e'\otimes \wp_{\frac{1}{2}\LL_{\omega_6}}, \quad F=f'\otimes \wp_{\frac{1}{2}\LL_{\omega_6}}^2,\quad H=h'.$$ By Theorem \ref{ClaLias}, this aLia is isomorphic to $\mathfrak{O}.$

\end{proof}
 \subsection{Summary}
In this section, we will summarise our results. The main results are firstly a classification of aLias on complex tori with base Lie algebra $\mathfrak{sl}_2$ and secondly, the normal forms (as defined in Definition \ref{def:NormalForm}) corresponding to each of the groups in Lemma \ref{lem:groups up to isomorphism}. In Table \ref{TableIntro} in the Introduction, we have given a classification in terms of the number of branch points of the canonical projection $\mathbb{T}\rightarrow \mathbb{T}/\Gamma$, where $\mathbb{T}$ is a punctured complex torus. It is a remarkable fact that the Onsager algebra appears only in the case whereby the complex torus has additional symmetry. We learn from Table \ref{TableIntro} that if $\Gamma\subset \A(T)$ is such that $\mathbb{T}\rightarrow \mathbb{T}/\Gamma$ has two branch points, then $(\mf{sl}_2\otimes_{\CC}\ot)^{\Gamma}$ is isomorphic to $\mf{O}$. Below we will give a more elaborate classification based on the possible symmetry groups, as we have found in Section \ref{MainResults}.\\

\textbf{ALias with $\Gamma=C_{\ell}\subset \A_0(T),\,\, \ell=2,3,4,6$}\\
The classification of aLias with $C_{\ell}$ ($\ell=3,4,6$) such that $C_{\ell}\subset \A_0(T)$, where $T$ is a suitable torus, has been done in Theorem \ref{ClaLias}. The normal forms are given in the proof of Theorem \ref{ClaLias} and \eqref{Normal Form C4}, \eqref{Normal Form C6}. In particular, for any faithful $\rho:C_{\ell}\rightarrow \A(\mathfrak{sl}_2)$ and $\sigma:C_{\ell}\rightarrow \A(T)$, 
the aLia $(\mathfrak{sl}_2\otimes_{\CC}\mathcal{O}_{\mathbb{T}})^{\rho\otimes\tilde{\sigma}(C_{\ell})}$ is isomorphic to the Onsager algebra.\\
The case of $\ell=2$ has been done in Example \ref{C2aLia} and is isomorphic to $\mf{S}_{\tau}$ for some $\tau\in \mathbb{H}$.\\

 \textbf{ALias with $\Gamma=C_{N}\subset t(T),\,\, N\in \ZZ_{\geq 1}$}\\
Fix a complex torus $T=\CC/\ZZ\oplus\ZZ\tau$ and let $\sigma_{\alpha}:C_N\rightarrow \A(T)$ be given by $\sigma_{\alpha}(r)z=z+\alpha$, where $\alpha$ is a $N$-torsion point of $T$. In Theorem \ref{CNaLia} we proved that for any $T$ it holds that $(\mathfrak{sl}_2\otimes_{\CC}\mathcal{O}_{\mathbb{T}})^{\rho\otimes\tilde{\sigma}_{\alpha}(C_N)}\cong \mathfrak{sl}_2\otimes_{\CC}\mathcal{O}_{\mathbb{T}}^{\tilde{\sigma}_{\alpha}(C_N)}$, regardless of the choice of $\rho:C_N\rightarrow\A(\mathfrak{sl}_2)$. Notice that $(\mathfrak{sl}_2\otimes_{\CC}\mathcal{O}_{\mathbb{T}})^{\rho\otimes\tilde{\sigma}_{\alpha}(C_N)}\cong \mathfrak{C}_{\tau'}$ for some $\tau'$ such that $T_{(\alpha)}\cong T_{\tau'}$, where we recall that $T_{(\alpha)}=\CC/(\ZZ+\ZZ\alpha+\ZZ\tau)$. \\

\textbf{ALias with $\Gamma=C_2\times C_2\subset t(T)$}\\
Theorem \ref{C2xC2aLia} says that for any faithful $\rho:C_{2}\times C_2\rightarrow \A(\mathfrak{sl}_2)$ and $\sigma:C_2\times C_2\rightarrow \A(T)$, where $T=\CC/\LL$, $(\mathfrak{sl}_2\otimes_{\CC}\mathcal{O}_{\mathbb{T}})^{\tilde{\sigma}(C_2\times C_2)}\cong \mathfrak{sl}_2\otimes_{\CC}\CC[\wp_{\frac{1}{2}\LL},\wp_{\frac{1}{2}\LL}']\cong \mathfrak{C}_{\frac{1}{2}\LL}$. Since the embedding of $C_2\times C_2$ inside $\A(T)$ is unique, we have that for each $\LL$, there is precisely one aLia with base Lie algebra $\mathfrak{sl}_2$ and symmetry group $C_2\times C_2$ embedded as translations. \\
 
\textbf{ALias with $\Gamma =A_4$}\\
By Theorem \ref{A4aLia}, $(\mathfrak{sl}_2\otimes_{\CC}\mathcal{O}_{\mathbb{T}})^{\rho\otimes\tilde{\sigma}(A_4)}\cong \mathfrak{O}$. In particular, this is independent of (faithful) $\sigma$ and $\rho$. \\

\textbf{ALias with $\Gamma=D_N, \,\, N\in \ZZ_{\geq 2}$}\\
In Corollary \ref{DNaLia}, we obtained $(\mathfrak{sl}_2\otimes_{\CC}\mathcal{O}_{\mathbb{T}})^{\rho\otimes\tilde{\sigma}_{\alpha}(D_N)}\cong \mathfrak{S}_{\tau}$ for some $\tau$ such that $T_{(\alpha)}\cong T_{\tau}$.
For $N=2$ we choose the first factor of $C_2\ltimes C_2$ to be a subgroup of $\A_0(T)$, as opposed to the third item in the above list ($C_2\times C_2\subset t(T)$). In this case $(\mathfrak{sl}_2\otimes_{\CC}\mathcal{O}_{\mathbb{T}})^{\rho\otimes\tilde{\sigma}_{\alpha}(D_2)}\cong \mathfrak{S}_{\tau}$ for some $\tau$.

\section{Discussion}
Our classification has been formulated in Theorem \ref{thm:classification}. This section provides a summary in terms of two tables, leading to Table \ref{TableIntro}, as given in the introduction. It will also highlight some directions for further research. \\

To explain how one arrives at Table \ref{TableIntro}, we will recall a basic fact. We have shown in Lemma \ref{lem:branch points} that the number of branch points of $\mathbb{T}\rightarrow \mathbb{T}/\Gamma$ equals 0 if $\Gamma\subset t(T)$. Furthermore, this number equals 2 if $\Gamma\in \{C_3,C_4,C_6,A_4\}$, where the cyclic groups are subgroups of $\A_0(T)$ for a suitable $T$.
For $\Gamma=D_N\cong C_2\ltimes C_N$ such that $C_2\subset \A_0(T)$, it holds that the number of branch points equals 3. We summarise our main result in Table \ref{Table3}.

\begin{table}[h]
        \begin{center}
            \begin{tabular}{c|ccc}
             &  0 & 2& 3        \\
            \hline\\[-4mm]
            $C_2$&  $\mathfrak{C}_{\tau}$ &   & $\mathfrak{S}_{\tau}$ \\
            $C_N,\,N=3,4,6$&  $\mathfrak{C}_{\tau}$ &  $\mathfrak{O}$ &  \\
            $C_N,\,N\neq 2, 3,4,6$&  $\mathfrak{C}_{\tau}$ &  &  \\
            $C_2\times C_2$& $\mathfrak{C}_{\tau}$ & & $\mathfrak{S}_{\tau}$ \\
            $D_N, \, N\ge 3$ & & &  $\mathfrak{S}_{\tau}$   \\
            $A_4$ & & $\mathfrak{O}$ &  \\ 
            \end{tabular}
        \caption{Isomorphism classes of aLias for each symmetry group $\Gamma$
         and each number of branch points of the quotient map $\mathbb{T}\rightarrow \mathbb{T}/\Gamma$.} 
        \label{Table3}
        \end{center}
    \end{table}
    The columns of Table \ref{Table3} are constant, so that we can summarise it further into Table \ref{TableIntro1}.
    \begin{table}[h!]
        \begin{center}
            \begin{tabular}{cc}
            $\#$ branch points &  Lie algebra       \\
            \hline\\[-4mm]
            0 & $\mathfrak{C}_{\tau}$  \\
            2 & $\mathfrak{O}_{\phantom{\tau}}$  \\
            3 & $\mathfrak{S}_{\tau}$
            \end{tabular}
        \caption{Lie algebra associated to the number of branch points of the quotient map $\mathbb{T}\rightarrow \mathbb{T}/\Gamma$.} 
        \label{TableIntro1}
        \end{center}
    \end{table}

We will present some possible directions for completing and extending the results in this paper. We know that $\mathfrak{C}_{\tau}\cong \mathfrak{C}_{\tau'}$ if and only if $[\tau]=[\tau']$. Settling how $\mf{S}_{\tau}$ breaks down into $\CC$-isomorphism classes would result in a more complete classification of our result. Notice that we can interpret this question also as to what extent the algebra determines the geometry in our problem.\\

An obvious extension would be to replace the base Lie algebra $\mf{sl}_2$ by $\g$, where $\g$ is a complex semisimple Lie algebra or even reductive $\g$. The challenge here is to construct intertwining maps; preliminary work shows that this is possible for some classes of higher dimensional base Lie algebras, but rather involved for others.\\

 This brings us to another question. The construction of intertwiners in Section \ref{MainResults} may appear to the reader to be somewhat ad-hoc. Can the main classification theorem, Theorem \ref{thm:classification}, be also proved without constructing bases, but only using the geometry of the problem, as suggested by Table \ref{TableIntro1}?\\
 
As we have pointed out in the introduction, aLias arose in the context of integrable systems. A recent paper by Mikhailov and Bury \cite{MR4187211} sparks renewed interest in the construction of Lax pairs using bases of aLias of the Riemann sphere. An interesting question is to consider what kind of integrable equations can be obtained using aLias on complex tori. In our Remark \ref{rem:LL}, we have highlighted a development arising in a different context leading to the Landau-Lifshitz hierarchy \cite{carey1993landau}. We hope that our classification allows for a systematic study of possible integrable systems associated to aLias on complex tori. This question is beyond the scope of this paper and it is research in progress.\\

Finally, we will make some remarks concerning the Onsager algebra. In this paper we have realised the Onsager algebra as an aLia in flat geometry. Previously, this algebra was realised in the spherical case \cite{roan1991onsager} and more recently also in hyperbolic geometry, cf. \cite{knibbeler2023automorphic}. It would be interesting to see how the Onsager algebra can be applied in a wider variety of geometries and what role the symmetry group plays in this context.\\

\textbf{Acknowledgements} We are grateful to Marta Mazzocco, Artie Prendergast, Jan Sanders and Alexander Veselov for very helpful and stimulating discussions. VK gratefully acknowledges support by the London Mathematical Society via an Emmy Noether fellowship. \\

\textbf{Funding} This work was supported by the Engineering and Physical Sciences Research Council (EPSRC):  the work of CO is supported by the grant [EP/W522569/1]; the work of SL and VK has been in part supported by the grant [EP/V048546/1].

\bibliographystyle{plain}

\end{document}